\documentclass[12pt]{amsart}
\usepackage{amssymb,amsmath}
\usepackage{geometry}    
\usepackage{mathrsfs}

\usepackage{vmargin}
\setmargrb{1in}{1in}{1in}{1in}

\newtheorem{theorem}{Theorem}[section]
\newtheorem{lemma}[theorem]{Lemma}
\newtheorem{proposition}[theorem]{Proposition}
\newtheorem{corollary}[theorem]{Corollary}
\theoremstyle{definition}
\newtheorem{definition}{Definition}
\newtheorem{remark}{Remark}
\newtheorem{example}{Example}

\newtheorem{problem}{Problem}

\begin{document}

\title[On intrinsic and extrinsic rational
approximation to Cantor sets]{On intrinsic and extrinsic rational
approximation to Cantor sets}

\author{Johannes Schleischitz}

\thanks{Middle East Technical University, Northern Cyprus Campus, Kalkanli, G\"uzelyurt \\
	johannes.schleischitz@univie.ac.at}

\begin{abstract} 
	We establish various new results on a problem proposed by K. Mahler
	in 1984 concerning rational approximation to fractal sets by rational numbers
	inside and outside the set in question, respectively. 
	Some of them provide a natural continuation and improvement
	of recent results of Broderick, Fishman and Reich and Fishman and Simmons. 
	A key feature is that many of our new results apply to more general, multi-dimensional fractal sets 
	and require only mild assumptions on the iterated function system.
	Moreover we provide a non-trivial lower bound for the distance of 
	a rational number $p/q$ outside the Cantor middle third set $C$ 
to the set $C$, in terms of the denominator $q$.  
We further discuss patterns of rational numbers in fractal sets.
We want to highlight two of them: Firstly, an upper bound for the number of rational (algebraic) numbers
in a fractal set up to a given height (and degree) for a wide class of fractal sets. Secondly we find properties of the denominator structure of rational points in ''missing digit'' Cantor sets, generalizing claims of Nagy
and Bloshchitsyn. 

\end{abstract}

\maketitle

{\footnotesize 
Math Subject Classification 2010: 28A80, 11H06, 11J13
}

\section{A question of Mahler and generalizations} \label{sec1}

\subsection{Introduction and notation} \label{se1}
In 1883, G. Cantor introduced what is today referred to as Cantor middle-third set. It consists of the real numbers in $[0,1]$ whose infinite
base $3$ representation 
avoids the digit $1$, i.e. numbers of the form
\[
\xi= \sum_{j\geq 1} \frac{w_{j}}{3^{j}}, \qquad\qquad w_{j}\in \{0,2\}.
\]
In 1984, K. Mahler~\cite{mahler} proposed the problem to study 
how well elements in this set can be approximated by rational numbers
within the set, and rational numbers outside of it.
Problems of this type are usually referred to as 
intrinsic and extrinsic approximation, respectively. 
He noticed that convergents of the continued fraction expansion
to numbers in the Cantor middle third set may
lie in the Cantor middle third set or not. 
His problem can be naturally generalized to wider families of fractal sets. 
The easiest twist 
is to take $b\geq 3$ an integer and a digit set
$W\subseteq \{ 0,1,\ldots,b-1\}$ of cardinality at least two and at 
most $b-1$, and to consider all numbers in $[0,1]$ whose base $b$ digits belong to $W$. 
We call such sets {\em missing digit Cantor sets}
and write $C=C_{b,W}$. 
We will consider wider classes of 
$d$-dimensional fractal sets, comprised in
the following Definition~\ref{deff}. 
 For convenience we always consider
$\mathbb{R}^{d}$ equipped with the supremum 
norm $\Vert \underline{z}\Vert= \max_{1\leq j\leq d} \vert z_{j}\vert$,  however all proofs below
can be easily modified if we work with 
the usual Euclidean norm instead.

\begin{definition}[IFS, Cantor set] \label{deff}
	We call a function $f:\mathbb{R}^{d}\to \mathbb{R}^{d}$ a {\em contraction}
	if for some fixed $0<\tau<1$ we have
	if 
	\[
	\Vert f(\underline{x})-f(\underline{y})\Vert\leq \tau \Vert \underline{x}-\underline{y}\Vert,
	\qquad\qquad \underline{x},\underline{y}\in\mathbb{R}^{d}.
	\]
	In this paper, an {\em iterated function system} (IFS) $F=(f_{1},\ldots,f_{J})$ 
	is a finite set of contractions.
	We call the IFS a {\em similarity IFS} if the contractions are similarities, 
	that is for $1\leq j\leq J$ there exist $0<\tau_{j}<1$ with the property
	\[
	\Vert f_{j}(\underline{x})-f_{j}(\underline{y})\Vert= \tau_{j} \Vert \underline{x}-\underline{y}\Vert,
	\qquad\qquad \underline{x},\underline{y}\in\mathbb{R}^{d}.
	\]
	We call the IFS {\em affine} if the contractions
	are affine functions
	on $\mathbb{R}^{d}$, i.e. functions
	\begin{equation} \label{eq:nf}
	f_{j}(\underline{y})= A_{j}\underline{y}+ \underline{b}_{j}, \qquad\qquad 1\leq j\leq J,
	\end{equation}
	where $A_{j}\in \mathbb{R}^{d\times d}$ 
	and $\underline{b}_{j}\in \mathbb{R}^{d}$.
	 We call an IFS {\em rational-preserving} if
	$f_{j}(\mathbb{Q}^{d})\subseteq \mathbb{Q}^{d}$ 
	for $1\leq j\leq J$.
	Any IFS induces a compact set $C\subseteq \mathbb{R}^{d}$, called
	attractor of the IFS, given as the unique solution of
	$\cup_{1\leq j\leq J} f_{j}(C)=C$. 
	We call any such $C$ a {\em Cantor set} and carry over the definitions above
	in the obvious way to Cantor sets (for example $C$ is called affine and 
	rational-preserving
	Cantor set if the corresponding IFS is affine and rational-preserving).
\end{definition}

We discuss relations between the notions of the above definition.
	Any similarity IFS is an affine IFS with matrices
	$A_{j}=S_{j}\cdot O_{j}$ for orthogonal
	matrices $O_{j}$ and $S_{j}=\tau_{j}I_{d}$ scalar multiples 
	of the identity matrix,
	see Hutchinson~\cite{hutchinson}. Hence any similarity 
	Cantor set is an affine Cantor set and for $d=1$ the concepts
	coincide.
	An affine IFS is rational-preserving if and only if 
	$A_{j}\in \mathbb{Q}^{d\times d}$
	and $\underline{b}_{j}\in \mathbb{Q}^{d}$ for $1\leq j\leq J$. 
	By the above characterization of $A_{j}$,
	to obtain a rational-preserving similarity Cantor set we require orthogonal matrices
	with rational entries (then we can choose $\tau_{j}\in\mathbb{Q}$). 
	A comprehensive description of all such matrices
	was given in~\cite{monthly}. Clearly,
	a special choice is $\tau_{j}\in\mathbb{Q}\cap (-1,1)$ and $O_{j}=I_{d}$ 
	for $1\leq j\leq J$. There are also various examples of non-affine, rational-preserving IFS.
	For $d=1$, one may consider $f(y)=f_{j}(y)$ any collection of rational functions 
	$cP(y)/Q(y)+y/2$ with
	$P,Q\in\mathbb{Q}[Y]$, with $Q$ not having real roots,
	degree of $P$ not exceeding the degree of $Q$, and rational
	$c>0$ sufficiently small to guarantee $\vert f^{\prime}(y)\vert\in[\epsilon,1-\epsilon]$ uniformly,
	for example $f(y)=y/(3+3y^{2})+y/2$. 
	It is further possible to construct real analytic,
	transcendental functions with this property (here transcendental
	means $f$ does not satisfy a polynomial identity $P(z,f(z))=0$ with
	$P\in\mathbb{C}[X,Y]$), it suffices to multiply the functions obtained by Marques and Moreira in~\cite[Theorem~1.2]{mamo}
	by any non-zero rational factor of absolute value less than $2/3$.

The existence of the unique fixed point $C$ in Definition~\ref{deff} 
follows from Banach's Fixed Point Theorem~\cite{hutchinson} applied with respect to
the Hausdorff metric on compact subsets of $\mathbb{R}^{d}$.
Any $\underline{\xi}\in C$ has an {\em address}
$(\omega_{1},\omega_{2},\ldots)\in\{1,2,\ldots,J\}^{\mathbb{N}}$, i.e. it can
be written in the form 
\[
\underline{\xi}= \lim_{i\to\infty} \pi_{1}\circ \cdots \pi_{i}(\underline{0}), \qquad\qquad \pi_{j}=f_{\omega_{j}}\in F.
\]
Addresses may not be unique, to guarantee uniqueness one typically
has to assume the so called
strong separation condition (SSC), see Definition~\ref{osc} below.
We do not dig deeper into this topic here.
We will further need the open set condition for some of our results.

\begin{definition}[OSC, SSC] \label{osc}
	An iterated function system satisfies the {\em open set condition} (OSC) if 
	there exists a bounded open set $O\subseteq \mathbb{R}^{d}$ so that  
	$ f_{j}(O)\subseteq O$ for $1\leq j\leq J$ and for
	$i\neq j$ we have $f_{i}(O)\cap f_{j}(O)=\emptyset$. For simplicity we will say 
	a Cantor set $C$ as in Definition~\ref{deff}
	satisfies OSC if its associated IFS does.
	A Cantor set $C$ satisfies the {\em strong separation condition} (SSC) if 
	the sets $f_{j}(C)$ are disjoint.
\end{definition}

We briefly discuss measures and dimension of fractals. 
Any Cantor set $C$ in Definition~\ref{deff} supports a natural probability
measure by Frostmann's Lemma (see~\cite{falconer}).
For similarity Cantor sets that 
satisfy OSC, it is just a multiple of some $\Delta$-dimensional Hausdorff measure, 
and
there is a well-known general formula to determine $\Delta$ in terms of the contraction factors~\cite{hutchinson}. See also~\cite{mauldin} for a generalization.
Clearly this value $\Delta\in[0,d]$ equals the Hausdorff dimension of $C$. 
For $C_{b,W}$ it just becomes
$\Delta=\log \vert W\vert/\log b$. We refer to 
Falconer~\cite{falconer} for an introduction to metric theory on fractals.
Rational approximation to Cantor type sets has been intensely
studied, especially metrical questions
with respect to the mentioned measure.
However, until recently, 
approximation to fractals in the sense of Mahler's question 
with restrictions of the rationals had not been
studied in detail.
Only in 2007 a first attempt was made by Levesley, Salp and Velani~\cite{leve}.
In 2011 a paper of Broderick, Fishman and Reich~\cite{1} dealt
with intrinsic approximation.
Two recent papers of Fishman and Simmons~\cite{fs,fishsimmons}
shed more light on the topic as well. We
rephrase important results from~\cite{1,fs, fishsimmons} in Section~\ref{s2} below.
The purpose of this paper is to establish further results on Mahler's question,
where possible in the very general settings of Definition~\ref{deff}.
In Section~\ref{se22} we generalize
results on intrinsic approximation from~\cite{1,fs}.
The focus of the paper lies on extrinsic approximation in Section~\ref{3}.
As a byproduct we obtain related results, especially
on the cardinality of rational/algebraic vectors of bounded degree and height
and their period lengths. Moreover we gain some new insight on the
structure of rational numbers in missing digit Cantor sets $C_{b,W}$. 
We gather all these findings in Section~\ref{4}.

\subsection{Recent results on Mahler's problem} \label{s2}

The main result of~\cite{1} 
shows a Dirichlet type result for intrinsic rational approximation
to missing digit
Cantor sets 
$C_{b,W}$. We present a slightly
simplified version
to avoid dealing with technical details.

\begin{theorem}[Broderick, Fishman, Reich (2011)]
	\label{known} 
	Let $C=C_{b,W}$ be a missing digit Cantor set and $\xi\in C$.
	If $\Delta=\log \vert W\vert/\log b$ denotes the Hausdorff dimension of $C$, 
	the inequalities
	\begin{equation} \label{eq:cl}
	1\leq q\leq b^{Q^{\Delta}}, \qquad 
	\vert \xi-\frac{p}{q}\vert \leq \frac{b}{Q q}
	\end{equation}
	have a solution $\frac{p}{q}\in\mathbb{Q}\cap C$ for every
	parameter $Q\geq 1$. In particular, if $\xi\notin \mathbb{Q}$, 
	there exist infinitely
	many $\frac{p}{q}\in\mathbb{Q}\cap C$ with the property
	\begin{equation} \label{eq:crux}
	\vert \xi-\frac{p}{q}\vert \leq \frac{1}{q (\log_{b} q)^{1/\Delta}}.
	\end{equation}
\end{theorem}

As observed in~\cite{1}, 
the bounds  \eqref{eq:cl} and
\eqref{eq:crux} become weaker when for fixed $b$ we extend the digit set
and thereby increase the Hausdorff dimension $\Delta$. 
This seems counter-intuitive as there are more
rational numbers in $C$, however in larger Cantor
sets there exist more irrational $\xi\in C$ to be approximated as well.    
A generalization of Theorem~\ref{known} was given by Fishman and Simmons~\cite{fs}.

\begin{theorem}[Fishman, Simmons (2014)] \label{argell}
	Let $C\subseteq \mathbb{R}$ be a one-dimensional, affine, rational-preserving 
	Cantor set, i.e. derived from an IFS as in \eqref{eq:hutt}, 
	that satisfies OSC. Let $\Delta$ be its Hausdorff dimension
	and let $\gamma$ be as in \eqref{eq:oppe}. 
	Let $\xi\in C$. Then 
	there is a constant $K$ such that for any $Q\geq \max_{1\leq j\leq J} q_{j}$ the
	estimate
	\[
	\vert \xi-\frac{p}{q}\vert \leq K q^{\gamma-1} (\log Q)^{-1/\Delta}
	\]	
	admits a solution $p/q\in \mathbb{Q}\cap C$ with $1\leq q\leq Q$.
\end{theorem}

For one-dimensional, affine, rational-preserving Cantor sets as in the theorem, $C$ is the
attractor of an iterated function system 
\begin{equation} \label{eq:hutt}
f_{j}(y)= \frac{p_{j}}{q_{j}} y+ \frac{r_{j}}{q_{j}}, \qquad\qquad 1\leq j\leq J,
\end{equation}
for $p_{j}/q_{j}$ of absolute value at
most $1$ and $r_{j}/q_{j}$ a rational number. 
We can assume $q_{j}>0$.
Then let
\begin{equation} \label{eq:oppe}
\gamma= \max_{1\leq j\leq J} \frac{\log \vert p_{j}\vert}{\log q_{j}}. 
\end{equation}
\begin{definition}
	If $d=1$, we call an affine, rational-preserving IFS {\em monic}  
	if $p_{j}\in\{1,-1\}$ in \eqref{eq:hutt} for all $1\leq j\leq J$, or equivalently if $\gamma=0$.
	We call the derived Cantor set a {\em one-dimensional, affine, monic, rational-preserving Cantor set}.
\end{definition}

Note that the missing Cantor sets $C_{b,W}$ defined above
are monic. Indeed the similarities
can be written $f_{j}(y)=y/b+w_{j}/b$ for $w_{j}\in W$, with
every contraction factor equal to $1/b$.
For monic Cantor sets as in the theorem 
where $\gamma=0$, the estimate becomes
\[
\vert \xi-\frac{p}{q}\vert \leq Kq^{-1} (\log Q)^{-1/\Delta}.
\] 
We thus identify Theorem~\ref{known} (up to the value of the constant)
as a special case.
We turn to extrinsic approximation.
A result
of Fishman and Simmons reads
as follows~\cite[Corollary~1.2]{fishsimmons}.

\begin{theorem}[Fishman, Simmons (2015)] \label{fs2}
	Let $C\subseteq \mathbb{R}^{d}$ be a Cantor set that satisfies OSC
	and does not
	contain a line segment. 
	Assume the additional property that for any 
	compact set $K\subseteq \mathbb{R}^{d}$ there is a constant $\kappa>0$ with
	\begin{equation} \label{eq:hypot}
	\Vert f_{j}(\underline{x})-f_{j}(\underline{y})\Vert\geq \kappa \Vert \underline{x}-\underline{y}\Vert,
	\qquad\qquad \underline{x}, \underline{y}\in K,\; 1\leq j\leq J.
	\end{equation} 
	Let 
	$\underline{\xi}\in C\setminus \mathbb{Q}^{d}$. Then for some 
	constant $c=c(C)>0$, there exist infinitely many
	$\underline{p}/q\in \mathbb{Q}^{d}\setminus C$ such that
	\[
	\Vert \underline{\xi}-\underline{p}/q\Vert \leq c\cdot q^{-1-\frac{1}{d}}.
	\]
\end{theorem}

Originally this result is only formulated for the 
narrower class of similarity Cantor sets in~\cite{fishsimmons}, 
however the given proof extends to our more general
situation, upon small modifiactions of the proof of~\cite[Lemma~2.12]{fishsimmons}.
Indeed, our local hypothesis \eqref{eq:hypot} on the contractions suffices to derive
$\gamma$ as in its proof, and the equality in the last displayed formula must
be altered to greater or equal, not affecting the implication of the lemma. 
Up to the constant $c$, the bound is of best possible 
order we can expect for generic $\underline{\xi}\in C$.
As observed in~\cite{fishsimmons}, if $d=1$, by Dirichlet's Theorem
the claim follows directly with $c=1$ 
whenever infinitely many convergents
to $\xi$ lie outside $C$. However, it was recently
shown~\cite{royschl} that
any missing Cantor set $C=C_{b,W}$ contains irrational
numbers $\xi$ 
with almost all convergents in $C_{b,W}$. The actual proof of Theorem~\ref{fs2}
employed a variant of Lemma~\ref{fslemma} below.

\section{Intrinsic approximation} \label{se22}

In this section we provide a variant of Theorem~\ref{argell} for Cantor sets in higher
dimension.

\begin{theorem} \label{ver}
	Let $C\subseteq \mathbb{R}^{d}$ be a rational-preserving affine Cantor set,
	i.e. the attractor of an IFS 
	$F$ consisting of contraction maps
	\[
	f_{j}(\underline{y})= \frac{A_{j}\underline{y}}{q_{j}}+\frac{\underline{b}_{j}}{s_{j}}, 
	\qquad\qquad 1\leq j\leq J, 
	\]
	where $A_{j}\in\mathbb{Z}^{d\times d}, \underline{b}_{j}\in\mathbb{Z}^{d}$
	and $q_{j},s_{j}\in\mathbb{N}$. Let $\tau_{j}\in(0,1)$ be the contraction factor of $f_{j}$ and $\tau=\max_{1\leq j\leq J} \tau_{j}$. Further let
	$S:=\prod_{1\leq j\leq J} s_{j}$ and
	\[
	\mu_{j}=\frac{\log \tau_{j}}{\log q_{j}},\qquad\qquad  \mu=\max_{1\leq j\leq J} \mu_{j}<0.
	\] 
	Then for any $\underline{\xi}\in C$ and any parameter $Q\geq S\cdot 2^{d}(\max_{1\leq j\leq J} q_{j})^{d}$,
	there exists $\underline{p}/q\in C\cap \mathbb{Q}^{d}$ with the properties
	\[ 1\leq q\leq Q, \qquad\qquad
	\Vert \underline{\xi}- \underline{p}/q\Vert \leq q^{\mu/d}(\log Q)^{\log \tau/\log J}.
	\] 
\end{theorem}

If $d=1$ and we identify $\gamma-1= \mu$ with $\gamma$ in \eqref{eq:oppe},
we almost obtain Theorem~\ref{argell}.
One difference is a slightly altered exponent for $\log Q$.
We believe the optimal exponent for $\log Q$ is again $-1/\Delta$ for $\Delta$
the Hausdorff dimension of $C$. For similarities $f_{j}$, our exponent is 
slightly worse, unless in the special case that all contraction
factors $\tau_{j}$ coincide where the identity $\log \tau/\log J=-1/\Delta$
can be readily verified~\cite{hutchinson}. In particular
we identify Theorem~\ref{known} as a special case of Theorem~\ref{ver}.
On the other hand, even for $d=1$ our Theorem~\ref{ver} is slightly more general 
than Theorem~\ref{argell} in the sense that we do not require $s_{j}=q_{j}$
as in \eqref{eq:hutt}, so $\mu/d=\mu$ is in general better than the inferred constant $\gamma-1$ obtained from transitioning to representation \eqref{eq:hutt}. 
Observe that $\mu/d\in[-1/d,0)$, with $\mu/d=-1/d$ for instance if all
$A_{j}$ are permutation matrices (with optional sign conversions $1$ to $-1$).
Note also that in contrast to~\cite{fs} we do not require OSC for the conclusion
of Theorem~\ref{ver}.

\section{Extrinsic approximation} \label{3}

\subsection{Uniform extrinsic approximation} \label{uni}

A topic that seems to be untouched so far is uniform extrinsic approximation
to Cantor sets. Theorem~\ref{yz}
below shows that under mild assumptions on the underlying IFS,
any element of
the derived Cantor set is uniformly approximable by rationals outside $C$
with denominator at most $Q$ of order $Q^{-1}$, and this bound is essentially optimal. Thereby,
we see that an improvement as for intrinsic approximation
in Theorem~\ref{known} and Theorem~\ref{argell} cannot be achieved for 
extrinsic approximation. Another interpretation is 
that the exponent $-1-1/d$ from Theorem~\ref{fs2} is not valid
when it comes to uniform approximation.
The result follows from a combination of
Theorem~\ref{thm110} and Theorem~\ref{thm11}, which are formulated 
in very general settings.

\begin{theorem} \label{thm110}
Let $C\subseteq \mathbb{R}^{d}$ be any Cantor set arising from
an IFS with contraction ratios $\tau_{j}, 1\leq j\leq J$ and let $\tau=\max_{1\leq j\leq J} \tau_{j}$
and $D:=-\log \tau/\log J>0$. Assume that either
\begin{itemize}
	\item $C$ satisfies the 
	OSC and \eqref{eq:hypot}, and there is a vector $\underline{v}\in \mathbb{Z}^{d}$
	such that $C$ contains no line segment parallel to $\underline{v}$, or 
	\item we have $D<1/2$
\end{itemize}

	Then there exists a constant $K=K(C)$ so that for every $\underline{\xi}\in C$ and
	every $Q\geq 1$, the inequality
	\begin{equation} \label{eq:opt11}
	\Vert \underline{\xi}-\underline{p}/q \Vert \leq \frac{K}{Q}
	\end{equation}
	has a solution 
	$\underline{p}/q\in\mathbb{Q}^{d}\setminus C$ with $1\leq q\leq Q$.
	\end{theorem}

We believe that the weak assumption on line segments just means that $C$
has empty interior, however for $d>1$ a proof of this would be desirable.
The implication from this first assumption 
is in fact a straight-forward consequence of results in~\cite{fishsimmons}.
The alternative latter assumption essentially says that $C$ has small Hausdorff
dimension, and the implication is a consequence of a new counting result in this case.
Now we turn towards the more challenging reverse estimates.

\begin{definition}
	We call a vector $\underline{v}\in \mathbb{R}^{d}$ {\em irrational} if it has
	at least one irrational coordinate, i.e. $\underline{v}\notin \mathbb{Q}^{d}$.
\end{definition}

\begin{theorem} \label{thm11}
	Let $C\subseteq \mathbb{R}^{d}$ be a rational-preserving Cantor set.
	Assume that there is $f_{j}\in F$ whose unique fixed point 
	$\underline{\alpha}_{j}$ lies in $\mathbb{Q}^{d}$. Assume further that  
	either 
	\begin{itemize}
		\item all contraction maps $f_{j}\in F$
		are one-to-one, or
		\item any element in $C$ has at most countably many addresses
		(which is true if $C$ satisfies the SSC).
	\end{itemize} 
	Let $\Phi: \mathbb{N}\to \mathbb{R}_{>0}$ be any function
	that tends to $0$ (arbitrarily
	slowly). Then the set of irrational $\underline{\xi}\in C$ for which
	\begin{equation} \label{eq:missing11}
	\Vert \underline{\xi}-\underline{p}/q \Vert > \frac{\Phi(Q)}{Q}
	\end{equation} 
	holds
	for infinitely many $Q\in\mathbb{N}_{>0}$ 
	and every $\underline{p}/q\in\mathbb{Q}^{d}\setminus C$
	with $1\leq q\leq Q$, is uncountable and dense in $C$. 
\end{theorem}

\begin{remark}
	Clearly the claim holds if $\underline{\xi}\in C\cap \mathbb{Q}^{d}$, and
	such vectors are dense in rational-preserving Cantor sets $C$
	(see also the last claim of Theorem~\ref{yz} and its proof below). Hence the irrationality 
	in the claim is	an important issue. Ideally we would like to sharpen the claim by asking the coordinates of $\underline{\xi}\in C$ to be $\mathbb{Q}$-linearly independent	together with $\{1\}$. We may infer this strengthened version 
	if any hyperplane
	in $\mathbb{R}^{d}$ intersects $C$ in at most countably many points, however 
	this assumption seems unnatural.
\end{remark}

\begin{remark}
	It is possible to relax the assumption that the 
	contraction maps $f_{i}$ are one-to-one in various ways. 
	However, we are unable to provide a very natural assumption and
	do not further elaborate on it here. The condition appears to be unrelated
	to OSC, and we believe OSC is not sufficient to imply the alternative 
	assumption either, however we are not aware of
	any concrete example. 
	See Sidorov~\cite{sidorov} for similarity Cantor sets
	in which all but finitely many elements have uncountably many addresses, which
	however do not satisfy the OSC. 
\end{remark}

While not satisfied in general,
the condition of the rational fixed point in Theorem~\ref{thm11} holds
if $C$ is a rational-preserving
{\em affine} Cantor set, i.e. derived from an IFS as
in \eqref{eq:nf}, as then clearly $\underline{\alpha}_{j}=-(A_{j}-I_{d})^{-1}\underline{b}_{j}\in\mathbb{Q}^{d}$
(note that $1$ is no eigenvalue of $A_{j}$ since it induces a contraction).
We may choose $A_{j}$ rational scalar multiples of the identity matrix
to obtain an IFS that consists of similarity contractions 
\begin{equation} \label{eq:hutte}
f_{j}(\underline{y})= \tau_{j}\underline{y}+\underline{b}_{j}, \qquad\qquad
\tau_{j}\in\mathbb{Q}\cap (-1,1)\setminus \{0\},\; \underline{b}_{j}\in\mathbb{Q}^{d}.
\end{equation}
Recall a Liouville number is an irrational real number with the property
that $\vert \xi-p/q\vert< q^{-N}$ has a rational solution $p/q$ for arbitrarily
large $N$. We call $\underline{\xi}\in\mathbb{R}^{d}$ Liouville vector if accordingly
$\vert \underline{\xi}-\underline{p}/q\vert< q^{-N}$ has infinitely many 
solutions for every $N$.
Theorem~\ref{thm11} is essentially
equivalent to asking that $C$ contains points 
that are arbitrarily well-approximable by rational vectors inside $C$ 
(in particular $C$ must contain 
''intrinsic Liouville vectors''). 
Indeed, the proof of Theorem~\ref{thm11} relies on the construction 
of such intrinsic Liouville vectors.
Combining these observations,
 from Theorem~\ref{thm110} and Theorem~\ref{thm11} we deduce at once

\begin{theorem} \label{yz}
	Let $C\subseteq \mathbb{R}^{d}$ be a
	rational-preserving similarity
	Cantor set, for instance derived from an IFS of contractions
	as in \eqref{eq:hutte}. Assume either
    OSC is satisfied and  
	 there is a vector $\underline{v}\in \mathbb{Z}^{d}$
	such that $C$ contains no line segment parallel to $\underline{v}$,
	or $D<1/2$ with $D$ as in Theorem~\ref{thm110}. Then there
	exists $K>0$ such that for any $\underline{\xi}\in C$
	and any $Q>1$ the inequality
	\[
	\Vert \underline{\xi}-\underline{p}/q\Vert \leq \frac{K}{Q}
	\]  
	has a solution $\underline{p}/q\in\mathbb{Q}^{d}\setminus C$ with $1\leq q\leq Q$. On the other hand, for 
	any function $\Phi: \mathbb{N}\to (0,\infty)$ that tends to $0$, there 
	exists $\underline{\xi}\in C\setminus \mathbb{Q}$ for which  
	\[
	\Vert \underline{\xi}-\underline{p}/q\Vert > \frac{\Phi(Q)}{Q}
	\]
  holds for certain arbitrarily large $Q$ and
  any $\underline{p}/q\in\mathbb{Q}^{d}\setminus C$ with $1\leq q\leq Q$.
  Finally, for $\Psi: \mathbb{N}\to (0,\infty)$ any function,
   there are intrinsically $\Psi$-approximable vectors in $C$, defined
  as the vectors $\underline{\xi}\in C$ for which the inequality
  \[
  \Vert \underline{\xi}-\underline{p}/q\Vert < \Psi(q)
  \]
  admits infinitely many solutions $\underline{p}/q\in\mathbb{Q}^{d}\cap C$.
  In particular, $C$ contains Liouville vectors.
\end{theorem}

As observed in Section~\ref{se1} a wider class of suitable IFS can be readily
found with the aid of~\cite{monthly}.
Any reasonable one-dimensonal
rational-preserving Cantor set with OSC, in particular any missing digit
Cantor set $C_{b,W}$ (with $W\subsetneq \{0,1,\ldots,b-1\}$), satisfies the assumptions 
of Theorem~\ref{yz}.
This first claim does not require the rationality of the IFS, however for
the other claims it is probably needed.
Intuitively, (similarity) Cantor sets
containing no Liouville vector should exist when we drop the rationality
condition.
For $d=1$ and the class of
topological Cantor sets, the existence was shown
in \cite{alvarez}. See also~\cite{erd} for fat Cantor sets without rational elements.
We remark further that for rational-preserving similarity 
Cantor sets induced by diagonal-matrices
$A_{j}$ that satisfy OSC, the existence of very well-approximable vectors 
that are no Liouville-vectors
was shown by S. Baker~\cite[Theorem~5.6]{baker}.
Notice also that the claim concerning Liouville vectors cannot be derived from standard metric results as 
in~\cite{fs} even for $C_{b,W}$, as the set of Liouville numbers
has Hausdorff dimension $0$ due to Jarn\'ik~\cite{jarnik}.

\subsection{Ordinary extrinsic approximation}

We use Theorem~\ref{known} and Theorem~\ref{argell} to 
obtain a lower bound on the distance of certain one-dimensional
Cantor sets
to a rational number not contained in it.
This can be viewed as a reverse
of Theorem~\ref{fs2}. For $A,B\subseteq \mathbb{R}$
we write $d(A,B)=\inf\{ \vert a-b\vert: a\in A, b\in B\}$ and denote by
$e=2.7182\ldots$ Euler's number.

\begin{theorem} \label{bfr} 
Let $C$ be a one-dimensional, affine, rational-preserving monic Cantor set, i.e. 
derived from an IFS as in \eqref{eq:hutt} with $p_{j}\in\{-1,1\}$, that
satisfies OSC.
Denote its Hausdorff dimension by $\Delta$.
Then for sufficiently large $\rho=\rho(C)$, we have the estimate
\begin{equation} \label{eq:000}  
d(C,\frac{p}{q}) > e^{-\rho q^{\Delta}}
\end{equation} 
for every $p/q\in \mathbb{Q}\setminus C$.

In the special case of missing digit Cantor sets $C=C_{b,W}$,
 we have $\Delta=\log \vert W\vert/\log b$ 
and the inequality
\begin{equation} \label{eq:null}
d(C,\frac{p}{q})> \frac{b^{-(2b)^{\Delta} q^{\Delta}}}{2q}
\end{equation}
holds for every $\frac{p}{q}\in \mathbb{Q}\setminus C$. 
In particular,
for every $\xi\in C_{b,W}$ and any 
$\delta>(2b)^{\Delta}$, the inequality
\begin{equation} \label{eq:0}  
\vert \xi-\frac{p}{q}\vert \leq b^{-\delta q^{\Delta}}
\end{equation} 
has only finitely many solutions 
$\frac{p}{q}\in \mathbb{Q}\setminus C$. 
\end{theorem}

We remark that 
in view of Theorem~\ref{ver}, it is possible to provide a variant of Theorem~\ref{bfr} 
that does not require OSC for the cost of (possibly) increasing $\Delta$ slightly.
We do not believe that the 
bounds are optimal. It might be true that there is an absolute upper
bound for the order of extrinsic approximation, equivalently
$\lambda_{ext}(\xi)\ll_{C} 1$ with the notation of Section~\ref{exponi} below.  
However, a metric result of Fishman and Simmons~\cite[Theorem~3.9]{fishsimmons}
demonstrates that we cannot hope an improvement 
of Theorem~\ref{known} and Theorem~\ref{argell} underlying our proof.
Hence the bounds in Theorem~\ref{bfr} are the limit of our method. 
Our proof does not extend to non-monic Cantor sets, as we require
that the right hand side in Theorem~\ref{argell}
decays faster than $q^{-1}$. This can only be guaranteed if 
$\gamma$ in \eqref{eq:oppe} vanishes. Also for $d>1$ we see that 
Cantor sets as in Theorem~\ref{ver}
do not fufill the requirement.
In these cases, no bound 
seems to be known. We formulate resulting open problems.

\begin{problem}
	Improve the bounds in Theorem~\ref{bfr}. Do there exist extrinsic Liouville 
	numbers, i.e. numbers with $\lambda_{ext}(\xi)=\infty$ in the notation of Section~\ref{exponi} below?
\end{problem}

\begin{problem}
Let $C$ be a one-dimensional, affine, 
rational-preserving Cantor set which is not necessarily monic. 
For $p/q\notin C$, find a lower bound
for $d(C,\frac{p}{q})$ in dependence of $q$. What about Cantor sets
in higher dimensional Euclidean space?	
\end{problem}

We believe the latter problem is related to Problem~\ref{problema} in
Section~\ref{se4} below. 
For $C=C_{b,W}$, an elementary approach, using that the base $b$ expansion
of a rational number $p/q$ in lowest terms
has period length $\ll q$, indicates the
bound
\begin{equation} \label{eq:null1}
d(C,\frac{p}{q})> b^{-cq}, \qquad\qquad
\text{for any} \quad \frac{p}{q}\in \mathbb{Q}\setminus C,
\end{equation}
with a suitable constant $c=c(b,W)>0$.
See Section~\ref{rel} below, in particular Proposition~\ref{pr}, for more details.
However, the conclusion \eqref{eq:null} is stronger than \eqref{eq:null1}
since $\Delta<1$.

We remark that similar patterns as in Theorem~\ref{bfr} 
are known concerning extrinsic rational
approximation to algebraic sets in $\mathbb{R}^{n}$.
See~\cite[Lemma~1]{bdl}, \cite[Lemma~1]{dd}, \cite[Lemma~4.1.1]{drutu},
\cite[Theorem~2.1]{ichmh}. 
However, the lower bounds typically decay like a negative power of $q$
in that case. In~\cite{ichmh}
it is shown that for an algebraic variety $S$ defined 
by an implicit integral polynomial equation of
total degree $k$, there are no rational numbers outside $S$ that
approximate $S$ of order greater than $k$. We refer to~\cite{4}
for approximation to manifolds by rationals within the manifold.

\section{Related topics} \label{4}

\subsection{Rational/Algebraic vectors in Cantor sets}

As a byproduct of the proof of Theorem~\ref{ver},
we show the following upper bound for the number
of rational elements with denominator
at most $N$ in a Cantor set. 

\begin{theorem} \label{yzok}
	Let $C\subseteq \mathbb{R}^{d}$ be any Cantor set as in Definition~\ref{deff}. 
	Let $\tau=\max_{1\leq j\leq J} \tau_{j}\in(0,1)$
	be the maximum contraction rate of the IFS and denote by $diam$ the 
	diameter $\max\{\Vert \underline{x}-\underline{y}\Vert: \underline{x},\underline{y}\in C\}$ of the compact Cantor set.
	Let $D=-\log J/\log \tau>0$. 
	Then the set of rational vectors in $C$ up to height $N$
	\[
	\mathcal{S}(C,N)= 
	\{ \underline{r}/s\in C: (r_{1},\ldots,r_{d},s)=1,\; 1\leq s\leq N\},
	\]
	has cardinality at most $J^{2}diam^{D}N^{2D}$.
\end{theorem}

We require $D<1$ or equivalently $J\tau<1$ for the result to be non-trivial.
When $C=C_{3,\{0,2\}}$ is the Cantor middle third set, the bound becomes
$4N^{2\log 2/\log 3}$. The exponent
has twice the expected magnitude, indeed in view of
numeric evidence it was conjectured
in~\cite{1} that $\mathcal{S}(C,N)\ll_{\epsilon} N^{\log 2/\log 3+\epsilon}$
for any $\epsilon>0$.
We can extend our claim to algebraic vectors.

\begin{theorem} \label{yok}
	With the assumptions and notation
	of Theorem~\ref{yzok}, let $\mathcal{S}(C,N,n)$ be the set of
	vectors in $C$ whose entries are real algebraic numbers of 
	degree at most $n$ and height 
	at most $N$. Then 
	\[
	\vert \mathcal{S}(C,N,n)\vert \ll_{C,n} N^{2nD}.
	\]
\end{theorem}

Clearly the trivial bound on $\mathcal{S}(C,N,n)$ is of order
$ N^{d(n+1)}$, so if $D\leq d/2$ or $J\tau^{d/2}<1$ we obtain an improvement simultaneously
for every $n\geq 1$. For the Cantor middle third set we get an improvement for
$n\in\{1,2,3\}$.
In case of affine, rational-preserving Cantor sets one might expect
$\mathcal{S}(C,N,n)=\mathcal{S}(C,N)$ for any $n\geq 1$, i.e. 
there are no irrational algebraic
vectors in $C$. For the Cantor middle-third set $C_{3,\{0,2\}}$ this is a famous
open conjecture of Mahler~\cite{mahler}. It seems that even in this case
no non-trivial bound on $\mathcal{S}(C,N,n)$ had been established previously. 

\subsection{Adresses of rational vectors in certain Cantor sets} \label{se4}

We want to discuss addresses of rational vectors in $C$, as the above sections
indicate that they are directly linked to the order of
rational approximation to Cantor sets. This topic has already been addressed
in \cite{fs} and
Theorem~\ref{pbasic} below generalizes~\cite[Lemma~4.2]{fs} to 
certain multi-dimensional settings. 
We also provide a quantitative version for period lengths. 
While our proof strategy resembles the one in~\cite{fs} to some extent,
we proceed slightly different. 

\begin{definition} \label{4def}
	We call an affine, rational-preserving IFS {\em unimodular} if the 
	contraction maps are of the form
 $f_{j}(\underline{y})=A_{j}\underline{y}/q_{j}+\underline{b}_{j}/s_{j}$
	where $A_{j}\in \mathbb{Z}^{d\times d}$ with $\det(A_{j})\in\{1,-1\}$, and $\underline{b}_{j}\in\mathbb{Z}^{d}$
	and $q_{j},s_{j}\in\mathbb{N}$. We call a Cantor set $C$ unimodular if it
	is induced by an unimodular IFS.
\end{definition}

Unimodularity generalizes the concept of a monic, affine IFS 
for $d=1$ from Section~\ref{s2}. Note that our setup is more general
than assuming $f_{j}(\underline{y})=(A_{j}\underline{y}+\underline{b}_{j})/q_{j}$, i.e. $s_{j}=q_{j}$, with 
all parameters as in the definition, as the determinant condition is relaxed.
Therefore our next theorem in fact even generalizes~\cite[Lemma~4.2]{fs} when $d=1$.

\begin{theorem} \label{pbasic}
	Let $C\subseteq \mathbb{R}^{d}$ be an affine, rational-preserving Cantor set.
	Then any vector in $C$ that admits an ultimately periodic address is in $\mathbb{Q}^{d}$.
	If $C$ is unimodular, the rational vectors 
	in $C$ are precisely those vectors in $C$ that have 
	an ultimately periodic address. 
	Moreover, in this case the period length (including preperiod) of
	$\underline{p}/q\in \mathbb{Q}\cap C$ can be chosen 
	$\ll_{C} \min\{q^{D}, q^{d}\}$ with $D=-\log J/\log \tau$ as in Theorem~\ref{yzok}. 
\end{theorem}

We remark that the bound supposedly can be reasonably improved, potentially up
to $\ll \log q$. See  the remark after the proof and also~\cite[Theorem~5.3, Conjecture~5.6]{fs}. 
For $d=1$,
the problem if the assumption of $C$ being unimodular (=monic) is necessary for the conclusion was already raised in~\cite[Section~5]{fs}. 
We have no new contribution to this interesting question, 
but want to formulate the corresponding generalization.

\begin{problem} \label{problema}
	In a Cantor set $C$ derived from an affine rational-preserving IFS, does {\em every}
	rational vector in $C$ have an ultimately periodic address? If yes, can it be
	arranged that the period length
	of $\underline{p}/q\in \mathbb{Q}\cap C$ is of order $\ll_{C} \min\{q^{D}, q^{d}\}$?
\end{problem}

From Theorem~\ref{pbasic} we get some information
on the structure of rational numbers in missing digit Cantor sets. We only highlight special consequences, from its proof below more information can be extracted.

\begin{corollary} \label{sophieg}
	Let $b\geq 3$ and $W\subsetneq \{0,1,\ldots,b-1\}$. 
	If $\mathcal{S}=\{q_{1},\ldots,q_{v}\}$ is any finite set of prime numbers 
	not dividing $b$, 
	then
	there are only finitely many rational numbers $r/s$ in $C_{b,W}$
	with $s$ consisiting only of prime factors in $\mathcal{S}$. In particular
	only finitely many integer powers of a rational number $p/q$ with $(b,q)=1$
	can lie in $C_{b,W}$.
	Moreover, $C_{b,W}$ contains at most finitely many rational numbers $p/q$ 
	where $q$ is a safe prime,
	i.e. $q$ and $(q-1)/2$ are both prime numbers.
\end{corollary}

It was recently pointed out to me by Michael Coons and Igor Shparlinki 
that it can be inferred from Korobov~\cite{korobov}
that the largest prime divisor of a denominator $s$
of $r/s\in C_{b,W}$ written in lowest terms is at least 
$\gg \sqrt{\log s \log\log s}$, implying the first two claims of the corollary.   
However, I was unable to find this deduction explicitly
in the literature. 
The method in~\cite{korobov} is based on exponential
sum estimates, unrelated to our proof below. 

Results of this type (restrictions on rationals in $C_{b,W}$) 
appear to be rare. Apart from~\cite{korobov}, the author is only aware of a succession of
papers by Wall~\cite{wall}, 
Nagy~\cite{nagy} and Bloshchitsyn~\cite{blo}, treating rationals with  
very smooth denominators. The first claim of Corollary~\ref{sophieg} 
extends~\cite[Theorem~2]{blo} (and thus also~\cite{nagy}) 
where the finiteness implication was proved for $\mathcal{S}=\{q\}$ a single prime
greater than $b^{2}$. For the second claim about powers of a rational
in fact we only require that the denominator $q$ has a prime divisor
that does not divide $b$, or equivalently the radical
of $q$ does not divide $b$. This condition is easily seen to be
necessary in general, as for example
any positive integer power of $1/3$ belongs to $C_{3,\{0,1\}}$.
For irrational numbers we state the analogous question as an open problem.

\begin{problem}
	Does there exist a real number $\xi$ which is not a root of a rational number and with infinitely many integral powers belonging to a missing digit Cantor set $C_{b,W}$?
\end{problem}

\subsection{Exponents} \label{exponi}

We define several exponents to measure the quality of intrinsic 
and extrinsic rational approximation in a Cantor set. We restrict to the
one-dimensional setting in this paper. In the sequel for convenience we agree
on $\sup(\emptyset)=0$ and $1/0=+\infty$.

\begin{definition}
	Let $S\subseteq \mathbb{R}$ and $\xi\in \mathbb{R}\setminus \mathbb{Q}$.
	Define the ordinary exponent of rational approximation
	$\lambda(\xi)$ as the supremum of $\lambda$ such that
	\begin{equation}  \label{eq:achh}
	\vert \xi-\frac{p}{q}\vert \leq q^{-1}Q^{-\lambda}
	\end{equation}
	has a solution $p/q\in \mathbb{Q}$ with $1\leq q\leq Q$ 
	for arbitrarily large values of $Q$. Define the uniform
	exponent $\widehat{\lambda}(\xi)$ as the supremum of $\lambda$ for which
	\eqref{eq:achh}
	for every large $Q$.
	Define similarly the ordinary and uniform exponents of
	intrinsic and extrinsic approximation 
	denoted by $\lambda_{int}(\xi), \lambda_{ext}(\xi)$
	and 
	$\widehat\lambda_{int}(\xi),\widehat{\lambda}_{ext}(\xi)$ respectively, 
	via replacing $p/q\in \mathbb{Q}$ by
	$p/q\in \mathbb{Q}\cap S$ and $p/q\in \mathbb{Q}\setminus S$ in the definitions
	accordingly.
\end{definition}

The exponents $\lambda(\xi), \widehat{\lambda}(\xi)$ are classical
exponents of Diophantine approximation and are usually denoted
by $\lambda_{1}(\xi)$ and $\widehat{\lambda}_{1}(\xi)$, respectively.
The following properties can be readily checked. We leave the verification
to the reader.

\begin{proposition}
	Let $S\subseteq \mathbb{R}$. 
	For any irrational $\xi\in S$ we have
	\begin{equation} \label{eq:inequa}
	\lambda(\xi)\geq \widehat{\lambda}(\xi)=1, \quad \lambda_{int}(\xi)\geq \widehat{\lambda}_{int}(\xi)\geq 0, \quad \lambda_{ext}(\xi)\geq \widehat{\lambda}_{ext}(\xi)\geq 0.
	\end{equation}
	Further we have
	\[
	\lambda(\xi)=\max \{ \lambda_{int}(\xi), \lambda_{ext}(\xi)\},
	\]
	and
	\begin{equation} \label{eq:cons}
	\max \{ \widehat{\lambda}_{int}(\xi), \widehat{\lambda}_{ext}(\xi)\}\leq \widehat{\lambda}(\xi)= 1.
	\end{equation}
\end{proposition}

\begin{remark}
	The
	ordinary exponent of approximation $\lambda(\xi)$ can equivalently be defined as the
	supremum of $\lambda$ such that
	\[
	\vert \xi-\frac{p}{q}\vert \leq q^{-1-\lambda}
	\]
	has infinitely many solutions $p/q\in \mathbb{Q}$. However, for the intrinsic
	and extrinsic exponents, for certain sets $S$ the according definition leads to different
	exponents than those defined above. 
	Several inequalities in \eqref{eq:inequa} may turn out false
	with this altered definition.  
\end{remark}

The identity $\widehat{\lambda}(\xi)= 1$ is due to Khintchine~\cite{khint}.
In general there is no equality in the last inequality.
From~\cite{fishsimmons,fs} quoted above some inequalities concerning special types
of Cantor sets can be inferred. For example, by Theorem~\ref{fs2} 
in any one-dimensional Cantor set $C$ 
that satisfies OSC we have
$\lambda_{ext}(\xi)\geq 1$
for any $\xi\in C\setminus \mathbb{Q}$.
Our main result is inspired by Theorem~\ref{thm11} and admits
a similar proof.

\begin{theorem} \label{expon}
	Let $S\subseteq \mathbb{R}$ be any set and $\xi\in S$. We have
	\begin{equation} \label{eq:e1}
	\widehat{\lambda}_{ext}(\xi)\leq \frac{1}{\lambda_{int}(\xi)}, \qquad
	\widehat{\lambda}_{int}(\xi)\leq \frac{1}{\lambda_{ext}(\xi)},
	\end{equation}
	If $S=C\subseteq \mathbb{R}$ is any one-dimensional Cantor set 
	that satisfies OSC 
	and $\xi\in C\setminus \mathbb{Q}$, then
	\begin{equation} \label{eq:foet}
	\widehat{\lambda}_{ext}(\xi)= \min\left\{ \frac{1}{\lambda_{int}(\xi)}, 1\right\}.
	\end{equation}
\end{theorem}

Define the spectrum of an exponent as the set of values taken 
when inserting any irrational real $\xi$. 
For missing digit Cantor sets $C_{b,W}$, Bugeaud~\cite{bug}
provided a construction of $\xi$ with $\lambda(\xi)=\lambda_{int}(\xi)=\tau$
for any given $\tau\geq 1$
in missing digit Cantor sets $C_{b,W}$. See also~\cite{leve},\cite{js}.
Therefore the spectrum of $\lambda$ in $C_{b,W}$ equals 
$[1,\infty]$, and the spectrum of $\lambda_{int}$ contains
this interval. Further in the case that $b$ is prime and $W$ contains
$0$ and $b-1$ but does not contain
two successive digits, the metrical result~\cite[Theorem~3.10]{fs}
implies that for $r\geq 0$ the set of $\xi\in C_{b,W}$ 
with $\lambda_{int}(\xi)=r$ (or $\lambda_{int}(\xi)\geq r$) has 
Hausdorff
dimension $\Delta/(r+1)$, where $\Delta=\log \vert W\vert/\log b$ is the Hausdorff dimension
of $C_{b,W}$. Thus 
from Theorem~\ref{expon} we deduce at once the following results. 

\begin{corollary} \label{cc}
	Consider the missing digit Cantor set $C_{b,W}$ of Hausdorff dimension $\Delta=\log \vert W\vert/\log b$.
	The spectrum of $\widehat{\lambda}_{ext}$ with respect to $C_{b,W}$ equals $[0,1]$. Assume $b$ is  prime and $W$ contains $\{0,b-1\}$ but does not contain
	two successive digits.
	Then for $\tau\in[0,1)$ the sets
	$\{\xi\in C_{b,W}:\widehat{\lambda}_{ext}(\xi)=\tau\}$ 
	and $\{\xi\in C_{b,W}:\widehat{\lambda}_{ext}(\xi)\leq \tau\}$
	have Hausdorff dimension $\Delta\cdot \frac{\tau}{\tau+1}$, whereas for $\tau=1$ the Hausdorff dimension equals $\Delta$.
	In particular the set $\{ \xi\in C_{b,W}: \widehat{\lambda}_{ext}(\xi)<1\}$ has 
	Hausdorff dimension $\Delta/2$.
\end{corollary}

We pose the problem to decide whether the claims of Corollary~\ref{cc} generalize
to one-dimensional Cantor sets that satisfy OSC. 
Finally we point out that numbers with the property $\widehat{\lambda}_{int}(\xi)=1$
do exist in any missing digit Cantor set $C_{b,W}$. 
Indeed, it is easy to verify this identity
for the numbers constructed in~\cite{royschl} which have almost all
convergents in $C_{b,W}$. The identity can further be checked 
for Liouville type numbers $\xi=\sum_{n\geq 0} wb^{-n!}$, $0\neq w\in W$, 
compare with~\cite[Lemma~3.10]{js}.
Thus, in view of the metric result rephrased above, we believe that the spectrum of $\widehat{\lambda}_{int}$ equals $[0,1]$ as well. However, we are unable to show this at present.

\subsection{Base expansions} \label{rel}

For missing digit Cantor sets
$C=C_{b,W}$, Theorem~\ref{pbasic} and
Theorem~\ref{bfr} have implications on the base $b$ digit patterns
of rational numbers that can be stated
without the framework of Cantor sets. Related considerations concerning the period lengths of
base $b$ representations of rational numbers in $C_{b,W}$ have been addressed 
in~\cite[Section~5]{fs}. However, the 
results there are almost all of conditional nature. We start with a consequence
of Theorem~\ref{pbasic}.  

\begin{theorem} \label{cons}
	Let $W\subseteq \{0,1,\ldots,b-1\}$ and $\Delta=\log \vert W\vert/\log b$. 
	Let $c_{0},\ldots,c_{N-1}\in W$ and assume that the infinite word
	$(c_{0}c_{1}\ldots c_{N-1})^{\infty}$ has period length $N$. Then we have
	\[
	\rm{gcd}(c_{0}b^{N-1}+c_{1}b^{N-2}+\cdots+c_{N-1}, b^{N}-1)\ll \frac{b^{N}}{N^{1/\Delta}}.
	\]
\end{theorem}

\begin{proof}
	Write $p/q$ for the fraction
	$(c_{0}b^{N-1}+c_{1}b^{N-2}+\cdots+c_{N-1})/(b^{N}-1)$ in lowest terms.
	Then the rational number $p/q$ has base $b$ expansion $(0.\overline{c_{0}c_{1}\ldots c_{N-1}})_{b}$, and by assumption 
	it has period length
	$N$, which by Theorem~\ref{pbasic} is $\ll q^{\Delta}$. Thus
	$q\gg N^{1/\Delta}$, which implies that the common factor is of size
	$\ll b^{N}/N^{1/\Delta}$. 
	\end{proof}

The natural assumption that the period length is not shorter 
than expected is necessary,
since if we let $c_{0}=c_{1}=\cdots=c_{N-1}\in W$ then
the gcd in question is a multiple of $(b^{N}-1)/(b-1)$ and thus $\gg b^{N}$.
Now we turn to implications of Theorem~\ref{bfr}.

\begin{theorem} \label{pthm}
Let $b\geq 3$ and $W\subseteq \{0,1,\ldots,b-1\}$. Further
let $\Delta=\log \vert W\vert/\log b$. 
Let
\begin{equation} \label{eq:ain}
\xi= (0.c_{0}c_{1}\cdots c_{k}\overline{c_{k+1}c_{k+2}\cdots c_{N-1}})_{b},
\end{equation}
with $N>k\geq 0$ and $c_{i}\in \{0,1,\ldots,b-1\}$ be a rational number in $(0,1)$
expanded in base $b$. Then $\xi=p/q$ with 
\begin{equation} \label{eq:asinas}
p=\sum_{i=0}^{N-1} c_{i}b^{N-1-i} -\sum_{j=0}^{k-1} c_{j}b^{k-1-j},\qquad
q=b^{N}-b^{k}. 
\end{equation}
Assume there exists an index $i$ with $c_{i}\notin W$ and
let $\phi(\xi)\in\{0,1,2,\ldots,N-1\}$ be the smallest index
with this property.
Further
let $p_{0}/q_{0}$ be the fraction $p/q$ in lowest terms.
Then
$\phi(\xi)\ll q_{0}^{\Delta}$.
Equivalently, if for $\xi$ in \eqref{eq:ain} the first
$s$ base $b$ digits $c_{0},c_{1},\ldots,c_{s}$ 
lie in $W$, but not all $c_{i}$
belong to $W$, then its reduced denominator $q_{0}$ is $\gg s^{1/\Delta}$.
\end{theorem}

The relation to Theorem~\ref{bfr}
arises from the fact that the rational numbers in $C_{b,W}$ are
precisely those where $c_{i}\in W$ in \eqref{eq:ain}.
For a generic rational $\xi$ in \eqref{eq:ain}
we expect that $q_{0}$ is of exponential
order, that is $\log q_{0}\gg N\geq s$, reasonably stronger than in our claim.
However, elementary methods only yield a bound  of order
$q_{0}\gg s$ that is valid for all $\xi$ in \eqref{eq:ain}, as we will
see from Proposition~\ref{pr} below. The bound 
$s^{1/\Delta}$ of our Theorem~\ref{pthm} is stronger since $\Delta<1$. 
The special case of rationals where $b$ is primitive
root of the denominator illustrates the improvement well. We denote by $\varphi$
the Euler totient function. For $A,B$ coprime integers, as usual
$ord_{A} \bmod B$ denotes the smallest positive integer $m$ with $A^{m}\equiv 1 \bmod B$.
The next proposition comprises results on the $\varphi$-function and period
lenghts of rational numbers in a base. 

\begin{proposition} \label{pr}
	We have $\varphi(q)\gg q/\log\log q$. 
	Let $A,B,C$ be integers with $(A,C)=1$. Write $C=C_{1}C_{2}$ with 
	$(C_{2},B)=1$ and $C_{1}$ consisting only of factors of $B$.
	Then $\xi=A/C$ written in base $B$ has preperiod of length $P(\xi)$
	equal to the smallest
	integer $v$ with $C_{1}\vert B^{v}$, followed by a period of length 
	$L(\xi)=ord_{B} \bmod C_{2}$. 
	In particular if $(A,B)=1$, then $A/C$ has purely periodic
   base $B$ expansion of period length $ord_{B} \bmod C$.
   Anyway, the total period length $P(\xi)+L(\xi)$ of $\xi$ is of order $\ll C$.
\end{proposition}

We omit a proof as the claims are well-known. The first claim can be found in the book of Hardy and Wright~\cite[Theorem~328]{hw}. The identities for periods can be checked in an elementary way.
The bounds on period lengths follow from $L(\xi)\leq ord_{B} \bmod C_{2}\leq \varphi(C_{2})\leq C_{2}-1\leq C-1$ and $P(\xi)\leq v\ll \log C_{1}\leq \log C$, or
alternatively directly from Theorem~\ref{pbasic}.
Recall $A$ is called primitive root modulo $B$
if $ord_{A} \bmod B=\varphi(B)$.
 
\begin{corollary} 
	Let $b,W, \Delta$ as in Theorem~\ref{pthm}.
	Let $p_{0},q_{0}$ be coprime integers and assume $b$ is a 
	primitive root modulo $q_{0}$. 
	Let \eqref{eq:ain} be the base $b$ expansion of $\xi=p_{0}/q_{0}$. 
	Assume not all $c_{j}$ belong to $W$ and let $i$ be the smallest index for which the digit $c_{i}$ lies outside $W$. 
	Then $i\ll N^{\Delta}(\log\log N)^{\Delta}$,
	i.e. $c_{i}\notin W$ for some $i\ll N^{\Delta}(\log\log N)^{\Delta}$ with an absolute
	implied constant.    
	In particular, any digit $v\in\{0,1,2,\ldots,b-1\}$ that occurs among the $c_{i}$ in \eqref{eq:ain} already occurs within the
	first $\ll N^{\log (b-1)/\log b}$ places.
\end{corollary}

\begin{proof}
	If $b$ is a primitive root modulo $q_{0}$ then by Proposition~\ref{pr}
	the number $\xi=p_{0}/q_{0}$ has no preperiod (i.e. $k=0$)
	and the period length $N$ satisfies $N=ord_{b}\bmod q_{0}=\varphi(q_{0})\gg q_{0}/\log\log q_{0}$. Thus,
	using the notation of Theorem~\ref{pthm}, from its claim we infer
	$i=\phi(\xi)\ll q_{0}^{\Delta}\ll (N \log\log N)^{\Delta}$, as desired.
	For the particular case, apply the above observation to $W=\{0,1,\ldots,b-1\}\setminus\{v\}$ of cardinality $b-1$, which induces $\Delta=\log(b-1)/\log b$.
\end{proof}

An equivalent way
to state the
claim of Theorem~\ref{pthm} is that numbers
$p,q$ as in \eqref{eq:asinas} have greatest common divisor
at most $\ll b^{N}/\phi(\xi)^{1/\Delta}$.
If we let $k=0$, 
another variant on the base $b$
expansion of large divisors of 
numbers of the form $b^{N}-1$ is obtained. 

\begin{corollary} \label{korolar}
Let $b, W$ and $\Delta<1$ as in Theorem~\ref{pthm}.
Let $\psi$ and $\phi$ be positive integers
with the property $\phi\geq c_{1}\psi^{\Delta}$ for
sufficiently large $c_{1}=c_{1}(b)>0$.
Assume for some integer $N\geq 1$ we
have $d^{\prime}\geq b^{N}/\psi$ divides $b^{N}-1$
and is written in base $b$ as
\begin{equation} \label{eq:letters}
d^{\prime}= (u_{0}u_{1}\cdots u_{N-1})_{b}= 
u_{0}b^{N-1}+u_{1}b^{N-2}+\cdots+u_{N-2}b+u_{N-1},
\end{equation} 
with possibly $c_{0}=c_{1}=\cdots=c_{s}=0$ for some $s\geq 0$, and assume all
$c_{i}\in W$.
Then for 
\begin{align*}
W_{1}= \{ u_{j}: 0\leq j\leq \phi\},\qquad 
\; W_{2}= \{ u_{j}: 0\leq j\leq N-1\},
\end{align*}
the letters occurring in the first $\phi$
and all digits of $d^{\prime}$, respectively,
we have $W_{1}=W_{2}$. 

In particular, if $\phi=\phi(N)=\lfloor rN^{\Delta}\rfloor$ 
for some $r\in(0,1]$,
and $\psi=\psi(N)=RN$ for 
sufficiently small $R=R(b,r)>0$, the following holds.
Assume $b^{N}-1$ has a small divisor $d_{N}\leq RN$,
and write the complementary divisor $d_{N}^{\prime}=(b^{N}-1)/d_{N}$
as in \eqref{eq:letters}.
The sets of digits that 
occur within $\{u_{0},u_{1},\ldots,u_{\lfloor N^{\Delta}r\rfloor}\}$
and 
$\{u_{0},u_{1},\ldots,u_{N-1}\}$ coincide.
\end{corollary}

In particular if we take $W$ of cardinality $b-1$, we see that 
for any large divisor of $b^{N}-1$, all occurring
digits in its
base $b$ expansion already can be found in a relatively
short initial sequence.
Generically, for large $N$ we expect 
$W_{1}=W_{2}=\{0,1,\ldots,b-1\}$. However,
$d_{N}^{\prime}=(b^{N}-1)/(b-1)$ 
in base
$b$ reads $(1^{N})_{b}$, so $W_{1}=W_{2}=\{1\}$. 
In this case we easily verify the claim. 
We end this section with an example.

\begin{example}
Let $b=3, W=\{0,1\}$ with $\Delta=\log 2/\log 3$. Then 
\[
\rm{gcd}(\epsilon_{0}3^{N-1}+\epsilon_{1}3^{N-2}+\cdots +\epsilon_{N-2}3^{1}+2, 3^{N}-1)\ll \frac{3^{N}}{N^{1/\Delta}},
\]
for any choice of $\epsilon_{j}\in\{0,1\}$, $0\leq j\leq N-2$ (observe that
$\epsilon_{N-1}=2\notin W$). Thus,
$3^{N}-1$ cannot have a very large divisor given by the sum expression. 
Consider special cases. First, assume
$\epsilon_{j}=0$ for $0\leq j\leq N-m$ and $\epsilon_{j}=1$
for $N-m+1\leq j\leq N-2$, for some integer $2\leq m\leq N+1$. Then the sum expression 
results in $(3^{m}+1)/2$, thus
\[
\rm{gcd}(\frac{3^{m}+1}{2}, 3^{N}-1)\ll \frac{3^{N}}{N^{1/\Delta}},\qquad \qquad 0\leq m\leq N-2.
\]
In particular the assumption
$(3^{m}+1)\vert (2\cdot (3^{N}-1))$
implies $3^{m}\ll 3^{N}/N^{1/\Delta}$, or 
equivalently $m<N-\log N/\log 2+c$
for $c\in\mathbb{R}$.
In this particular example,
in fact the divisibility condition 
implies the stronger (sharp) estimate $m\leq N/2$,
since $(3^{m}+1)\vert (2\cdot (3^{N}-1))$
is equivalent to $(3^{m}+1)\vert (2\cdot (3^{N-m}+1))$.
As a second special case let $\epsilon_{j}=1$ precisely for $j$ a power of $2$
and $\epsilon_{j}=0$ otherwise.
We are unable to find an elementary proof that confirms our implication
\[
\rm{gcd}(3^{2^{k}}+3^{2^{k-1}}+\cdots +
3^{2^{0}}+2, 3^{N}-1)\ll \frac{3^{N}}{N^{1/\Delta}}, \qquad k<\frac{\log N}{\log 2}.
\]
\end{example}

\section{Proofs}

\subsection{Proof of Theorem~\ref{ver}}

The crucial step
for the proof of Theorem~\ref{ver} is to extend~\cite[Lemma~2.2]{fs}.
Fortunately, this can be derived rather easily. Compared to~\cite{fs}, it does not involve the measure supported
on $C$, which might not be nice in our general setting,
but stems from the following elementary counting argument.

\begin{proposition} \label{ppr}
	Let $C$ be any Cantor set as in Definition~\ref{deff} and let 
	$\tau=\max_{1\leq j\leq J} \tau_{j}$ be the maximum contraction ratio of the IFS.
	 If $l\geq 1$ is an integer and
	$E$ a subset of $C$
	of cardinality greater than $J^{l}$, then there are two elements
	$\underline{x},\underline{y}$ in $E$ with distance  
	$\Vert \underline{x}-\underline{y}\Vert\leq diam \cdot \tau^{l}$, where $diam$ is the diameter of the compact
	set $C$. 
\end{proposition}

\begin{proof}
	For any $\underline{\alpha} \in E$
	let $(\omega_{1,\underline{\alpha}},\omega_{2,\underline{\alpha}},\ldots)$
	be any address of $\underline{\alpha}$ and let $\pi_{j,\underline{\alpha}}=f_{\omega_{j,\underline{\alpha}}}$.
	Then by pigeon hole principle since $\vert E\vert>J^{l}$
	there are two elements $\underline{x},\underline{y}$ 
	whose prefices $(\omega_{1,\underline{x}},\ldots,\omega_{l,\underline{x}})$
	and $(\omega_{1,\underline{y}},\ldots,\omega_{l,\underline{y}})$
	coincide, so $\pi_{j}:=\pi_{j,\underline{x}}=\pi_{j,\underline{y}}$ for $1\leq j\leq l$. 
	Let
	\[
	\underline{a}=\lim_{n\to\infty} \pi_{l+1,\underline{x}}\circ \pi_{l+2,\underline{x}}\cdots\circ \pi_{l+n,\underline{x}}(\underline{0}),\qquad
	\underline{b}=\lim_{n\to\infty} \pi_{l+1,\underline{y}}\circ \pi_{l+2,\underline{y}}\cdots\circ \pi_{l+n,\underline{y}}(\underline{0}).
	\]
	Clearly $\underline{a},\underline{b}\in C$. Then we have
	\[
	\Vert \underline{x}-\underline{y}\Vert=
	\Vert \pi_{1}\circ \pi_{2}\circ \cdots \circ \pi_{l}(\underline{a})- \pi_{1}\circ \pi_{2}\circ \cdots \circ \pi_{l}(\underline{b})\Vert
	\leq \tau^{l}\Vert \underline{a}-\underline{b}\Vert.
	\]
Since  $\underline{a},\underline{b}\in C$ their distance is bounded above by $diam$
and the claim follows. 
\end{proof}

From the proposition we immediately obtain the required
variant of~\cite[Lemma~2.2]{fs}.

\begin{lemma} \label{hoben}
	Let $C,\tau$ be as in Proposition~\ref{ppr}. Then,
	if $N\geq 1$ is an integer and
	$\underline{\xi}_{1},\ldots,\underline{\xi}_{N}$ belong to $C$,
	there is constant $K_{1}>0$ and two indices $1\leq i<j\leq N$ with 
	\[
	\Vert \underline{\xi}_{i}-\underline{\xi}_{j}\Vert \leq 
	r_{N}:= (N/K_{1})^{\log \tau/\log J}.
	\]
	We may choose $K_{1}=(diam/\tau)^{-\log J/ \log\tau}$. 
\end{lemma}

\begin{proof}
	For given $N$ choose the integer $l\geq 0$ so that $J^{l}<N\leq J^{l+1}$.
	By Proposition~\ref{ppr} there are two elements in the sequence with distance
	at most $diam\cdot \tau^{l}\leq diam\cdot \tau^{\log N/\log J-1}= (diam/\tau)\cdot N^{\log \tau/\log J}$, which leads to $K_{1}$ in the theorem.
	\end{proof}

In the sequel we denote by $\Vert A\Vert_{\infty}:= \max_{\underline{x}\neq \underline{0}} \{\Vert A_{j}\underline{x}\Vert/\Vert\underline{x}\Vert\}$ the norm of a matrix $A\in\mathbb{R}^{d\times d}$.

\begin{proof}[Proof of Theorem~\ref{ver}]
	We proceed essentially as in the proof of~\cite[Theorem~2.1]{fs}.
	Let $\underline{\xi}\in C$ arbitrary.
	Fix an address $\omega=(\omega_{1},\ldots)\in \{1,2,\ldots,J\}^{\mathbb{N}}$ of it, 
	so that with $\pi_{j}=f_{\omega_{j}}$ we have
	$\underline{\xi}=\lim_{k\to\infty}\pi_{1}\circ \pi_{2}\circ \cdots \pi_{k}(\underline{0})$. If $\sigma$
	is the left shift on the space of infinite formal words, for every large
	$N$ we consider 
	$\{\sigma^{n}(\omega)(\underline{0}): 0\leq n\leq N\}$, a finite sequence
	in $C$.
	By Lemma~\ref{hoben} there are two integers $0\leq n< m+n\leq N$ so that if
	\begin{align*}
	\underline{y}&:=\sigma^{n}(\omega)(\underline{0})=\lim_{k\to\infty}\pi_{n+1}\circ \cdots\circ\pi_{n+k}(\underline{0}), \\ \underline{z}&:=\sigma^{m+n}(\omega)(\underline{0})=\lim_{k\to\infty}\pi_{m+n+1}\circ \cdots\circ\pi_{m+n+k}(\underline{0}),
	\end{align*}
	then
	$\Vert \underline{y}-\underline{z} \Vert\leq r_{N}$.
	Define endomorphisms on $\mathbb{R}^{d}$ by
	$u_{(1)}=\pi_{1}\circ \cdots \circ \pi_{n}$ and
	$u_{(2)}=\pi_{n+1}\circ \cdots \circ \pi_{m+n}$, so that
	$\underline{\xi}= u_{(1)}(\underline{y})$ and $\underline{y}=u_{(2)}(\underline{z})$. Consider  $A_{\omega_{j}}\in\mathbb{Z}^{d\times d}$
	and $q_{\omega_{j}}\in\mathbb{Z}$ as in the theorem. Then let
	\begin{align*}
	P_{(1)}&= A_{\omega_{1}}\cdots A_{\omega_{n}}, \qquad P_{(2)}= A_{\omega_{n+1}}\cdots A_{\omega_{n+m}}  \\
	q_{(1)}&= q_{\omega_{1}}\cdots q_{\omega_{n}}, \qquad\quad q_{(2)}= q_{\omega_{n+1}}\cdots q_{\omega_{n+m}}.
	\end{align*}
	Further with $S:= \prod_{1\leq j\leq J} s_{j}$ define the integer vectors
	\[
	\underline{r}_{(1)}= \sum_{i=1}^{n} A_{\omega_{1}}\cdots A_{\omega_{i-1}}\underline{r}_{\omega_{i}}q_{\omega_{i+1}}\cdots q_{\omega_{n}}, 
	\]
	and
	\[
	\underline{r}_{(2)}= \sum_{i=1}^{m} A_{\omega_{n+1}}\cdots A_{\omega_{n+i-1}}\underline{r}_{\omega_{n+i}}\frac{S}{s_{\omega_{n+i}}}q_{\omega_{n+i}}\cdots q_{\omega_{n+m}}.
	\]
	In $r_{(2)}$ the additional
	factor $S/(q_{\omega_{n+1}}s_{\omega_{n+i}})$ compared to~\cite{fs} has entered according
	to the notational difference $s_{j}\neq q_{j}$ for the shift vector. 
	By the recursive process, the maps $u_{(i)}$ are accordingly given as
	\[
	u_{(i)}(\underline{t})= \frac{P_{(i)} \underline{t}}{q_{(i)}}+\frac{\underline{r}_{(i)}}{Sq_{(i)}}, \qquad\qquad i=1,2.
	\]
	The unique fixed point of $u_{(2)}$ denoted by $\underline{F}_{2}$ is given by the solution of
	\[
	\underline{F}_{2}=\frac{P_{(2)}}{q_{(2)}}\underline{F}_{2}+ \frac{\underline{r}_{(2)}}{Sq_{(2)}},
	\]
	which yields
	\[
	\underline{F}_{2}= \frac{1}{S}(q_{(2)}I_{d}-P_{(2)})^{-1}\underline{r}_{(2)}.
	\]
	We carry out why
	the inverse matrix is well-defined. All $A_{j}$ have norm 
	$\Vert A_{j}\Vert_{\infty}<q_{j}$ since $A_{j}/q_{j}$ are contractions for $1\leq j\leq J$, thus their product
	which is $P_{(2)}$ has norm less than $q_{(2)}$, i.e.
	\begin{equation} \label{eq:infty}
	\Vert P_{(2)}\Vert_{\infty} < q_{(2)}.
	\end{equation}
	Consequently all eigenvalues of $P_{(2)}$ are of absolute value
	smaller than $q_{(2)}$, proving the regularity of $q_{(2)}I_{d}-P_{(2)}$.
	We also observe that $\underline{F}_{2}\in\mathbb{Q}^{d}$ by Cramer's rule,
	hence this applies to $u_{(1)}(F_{2})$ as well
	so we may write
	\[
	\underline{p}/q=u_{(1)}(F_{2})=
	\frac{P_{(1)}(q_{(2)}I_{d}-P_{(2)})^{-1}\underline{r}_{(2)}
		+\underline{r}_{(1)}}{Sq_{(1)}},
	\]
	with $\underline{p}\in\mathbb{Z}^{d}$ and $q\in\mathbb{N}$. 
	Concretely by Cramer's rule the inverse matrix contributes a factor $\det (q_{(2)}I_{d}-P_{(2)})>0$ in the denominator, so the total denominator will be
	$q=Sq_{(1)}\det (q_{(2)}I_{d}-P_{(2)})$ if we do not simplify the vector $\underline{p}/q$ to lowest terms. 
	For any matrix $A\in\mathbb{R}^{d\times d}$ 
	we have $\vert\det A\vert\leq \Vert A\Vert_{\infty}^{d}$ (determinant
	is at most product of column norms, and columns are images of canonical base vectors), and from
	\eqref{eq:infty} we infer $\Vert q_{(2)}I_{d}-P_{(2)}\Vert_{\infty}\leq \Vert q_{(2)}I_{d}\Vert_{\infty}+\Vert P_{(2)}\Vert_{\infty}\leq
	2q_{(2)}$. Applied to $A=q_{(2)}I_{d}-P_{(2)}$
	we derive
	\begin{equation} \label{eq:ee}
	0<q\leq Sq_{(1)}\cdot (2q_{(2)})^{d}\leq S(2q_{(1)}q_{(2)})^{d}.
	\end{equation}
	Taking logarithms yields
	\begin{equation} \label{eq:tier}
	\log q\leq d \sum_{i=1}^{m+n} \log q_{\omega_{i}}+c,
	\end{equation}
with the constant $c=d\log 2+\log S$. Let $\tau_{(i)}$ be the contraction
	rates on $u_{(i)}$, $i=1,2$. Then, on the other hand
	\[
	\Vert \underline{y}-\underline{F}_{2}\Vert \leq \tau_{(2)}\Vert \underline{z}-\underline{F}_{2}\Vert \leq
	\tau_{(2)}\Vert \underline{y}-\underline{F}_{2}\Vert+\tau_{(2)}\Vert \underline{z}-\underline{y}\Vert,
	\]
	hence
	\[
	\Vert \underline{y}-\underline{F}_{2}\Vert \leq \frac{\mu_{(2)}}{1-\mu_{(2)}}\Vert \underline{z}-\underline{y}\Vert.
	\]
	Applying $u_{(1)}$ gives
	\[
	\Vert \underline{\xi}-\underline{p}/q\Vert \leq \frac{\tau_{(1)}\tau_{(2)}}{1-\tau_{(2)}} \Vert \underline{z}-\underline{y} \Vert
	\leq \frac{\tau_{(1)}\tau_{(2)}}{1-\tau_{(2)}} r_{N}.
	\]
	Now $\tau_{(2)}\leq \max_{1\leq j\leq J} \tau_{j}=\tau$, thus with $K_{2}= (1-\tau)^{-1}$ we infer
	\[
	\log \Vert \underline{\xi}-\underline{p}/q\Vert \leq \log K_{2}+\log r_{N}+
	\mu\sum_{i=1}^{m+n} \log \tau_{\omega_{i}}+c.
	\]
	By assumption $\log \tau_{\omega_{i}}=\mu_{\omega_{i}}\log q_{i}\leq \mu \log q_{i}$
	for every $1\leq i\leq m$. Thus from \eqref{eq:tier} and since $\mu\leq 0$ further
	\[
	\log \Vert \underline{\xi}-\underline{p}/q\Vert \leq \log K_{2}+\log r_{N}+
	\frac{\mu}{d} \log q+ \frac{c}{d},
	\]
	and after
	exponentiating again we get
	\begin{equation} \label{eq:fromc}
	\Vert \underline{\xi}-\underline{p}/q\Vert 
	\leq  K_{3}r_{N}q^{\mu/d},
	\end{equation}
	with some constant $K_{3}$. 
	On the other hand, again from \eqref{eq:ee} we derive
	\[
	\log q\leq d\sum_{i=1}^{m+n} \log q_{\omega_{i}}+d\log 2+\log S.
	\]
	Let $q_{max}=\max_{1\leq j\leq J} q_{j}$.
	Finally since $q\leq  S\cdot 2^{d}q_{max}^{d(m+n)}\leq  S\cdot 2^{dN}q_{max}^{dN}= S\cdot (2q_{max})^{dN}$, for 
	$N:= \lfloor \log Q/(d(\log (2q_{max}))+\log S)\rfloor$ and $Q\geq S2^{d}q_{max}^{d}$ we conclude $q\leq Q$ and 
	\[
	r_{N}\leq \left(\frac{N}{K_{1}}\right)^{\log \tau/\log J}\leq K_{4}
	\left(\frac{\log Q}{2dK_{1}\log q_{max}}\right)^{\log \tau/\log J}
	=K_{5} \log Q^{\log \tau/\log J}.
	\]
	Putting $K_{1}=K_{3}K_{5}$ we derive the claim from \eqref{eq:fromc}. 
\end{proof}

\subsection{Proofs from Section~\ref{3}}

For the proof of Theorem~\ref{thm110} we utilize an auxiliary result established
within the proof of~\cite[Theorem~1.1]{fishsimmons}. It deals
with almost-arithmetic sequences in Cantor sets. While
originally formulated only for similarity Cantor sets, as
pointed out in Section~\ref{s2} we may extend its claim to
Cantor sets as in Definition~\ref{deff}. 
We do not rephrase the notion of almost-arithmetic from~\cite{fishsimmons} and restrict to the arithmetic sequence setting that suffices for our proof.  

\begin{lemma}[Fishman, Simmons] \label{fslemma}
	Let $C\subseteq \mathbb{R}^{d}$ 
	be a Cantor set which satisfies OSC. There
	exists a positive integer $N=N(C)$ with the property
	that if set $C$ contains
		an arithmetic progression of length $N$, then the entire line
		segment joining these points is contained in $C$. In particular,
		if $C$ contains no line segment, then no non-constant
		arithmetic progression of length $N$ is contained in $C$.
\end{lemma}

The lemma will apply to the first condition of Theorem~\ref{thm110}.
For the implication from the latter condition, we will employ 
our counting result Theorem~\ref{yzok}

\begin{proof}[Proof of Theorem~\ref{thm110}] 
	As pointed out the claim is almost an immediate consequence of Lemma~\ref{fslemma}. 
	In our situation, given $\underline{\xi}\in C$ and $Q>1$, we consider $\underline{p}/\lfloor Q\rfloor$ 
	for $\underline{p}=(p_{1},\ldots,p_{d})$ chosen such that
	$p_{i}/\lfloor Q\rfloor$ is the rational number with denominator
	$\lfloor Q\rfloor$ that is closest 
	to the $i$-th coordinate
	$\xi_{i}$ of $\underline{\xi}\in C$. Clearly $\Vert \underline{\xi}-\underline{p}/\lfloor Q\rfloor\Vert \ll Q^{-1}$. 
	Assume the first condition of the theorem.
	Let $N=N(C)$ be large enough that $C$ contains no
	arithmetic sequence of length $N$ as in Lemma~\ref{fslemma}. Then
	with $\underline{v}\in\mathbb{Z}^{d}$ as in the theorem,
	the rational vectors
	\[
	\underline{p}/\lfloor Q\rfloor,\; \underline{p}/\lfloor Q\rfloor+\frac{1}{\lfloor Q\rfloor}\underline{v},\;
	\underline{p}/\lfloor Q\rfloor+\frac{2}{\lfloor Q\rfloor}\underline{v},\;
	\ldots,\; \underline{p}/\lfloor Q\rfloor+\frac{N}{\lfloor Q\rfloor}\underline{v}
	\]
	have common denominator at most $\lfloor Q\rfloor\leq Q$ and form an arithmetic progression. Thus
	not all can belong to $C$ as otherwise by Lemma~\ref{fslemma} the line segment joining
	$\underline{p}/q$ and $\underline{p}/\lfloor Q\rfloor+N/\lfloor Q\rfloor\cdot \underline{v}$ would be contained
	in $C$, contradicting our hypothesis. On the other hand we easily see
	that any element of the progression has distance
	$\ll_{N,\underline{v}} Q^{-1}\ll_{C,\underline{v}} Q^{-1}$ 
	from $\underline{\xi}$. 
	
	Finally, assume the second condition $D<1/2$.
	Then for every prime number $N$ with $Q/2\leq N<Q$ 
	again let $\underline{p}_{N}/N$ be the rational
	vector with numerator coordinates $p_{N,i}$ chosen so that
	$p_{N,i}/N$ is closest to $\xi_{i}$. We see
	that $\Vert \underline{p}_{N}/N-\underline{\xi}\Vert\ll Q^{-1}$.
	Moreover, for $Q>diam$ these vectors are clearly pairwise distinct
	and by Prime Number Theorem there are $\gg Q/\log Q$ 
	such vectors. On the other hand, by Theorem~\ref{yzok} there are only 
	$\ll Q^{2D}$ rational vectors in $C$ with denominator at most $Q$.
	Thus if $D<1/2$, for large $Q$ some must lie outside $C$ and satisfy
	the desired property.
\end{proof}

The proof of Theorem~\ref{thm11} requires more preparation.
The next lemma comprises some estimates on continued fractions and is partly well-known.

\begin{lemma} \label{lele}
	Let $\xi$ be a real number with sequence of 
	convergents $(u_{t}/v_{t})_{t\geq 1}$.
	Then
	\begin{equation}  \label{eq:formu}
	\frac{1}{2v_{t}v_{t+1}}\leq \vert \xi-\frac{u_{t}}{v_{t}}\vert \leq \frac{1}{v_{t}v_{t+1}}, \qquad\qquad t\geq 1.
	\end{equation}
	Moreover, if $t\geq 2$ is given and
	$r/s\neq u_{t}/v_{t}$ is any rational number with $0<s<v_{t+1}/2$,
	then
	\begin{equation}  \label{eq:formu2}
	\vert s\xi-r\vert \geq
	\frac{1}{2v_{t}}.
	\end{equation}
\end{lemma}

For the proof recall that for an origin-symmetric convex set $K\subseteq \mathbb{R}^{h}$ and a lattice $\Lambda\subseteq \mathbb{R}^{h}$, the $i$-th
successive minimum  $\lambda_{i}(K,\Lambda)$ of $\Lambda$ with
respect to $K$ is defined as the infimum
of real numbers $\lambda$ such that
$\Lambda\cap K$ contains $i$ linearly independent vectors (for $1\leq i \leq h$). Minkowski's Second Convex
Body Theorem then bounds the product of all successive minima
in terms of the volume of $K$, the fundamental area of the lattice $\Lambda$ and
$h$. In particular the case $h=2$ we will require reads as
\begin{equation}  \label{eq:diesegl}
 \frac{2\det \Lambda}{V(K)}\leq \lambda_{1}(K,\Lambda)\lambda_{2}(K,\Lambda)
 \leq \frac{4\det \Lambda}{V(K)}.
\end{equation}

\begin{proof}[Proof of Lemma~\ref{lele}]
	The inequalities of \eqref{eq:formu} are in fact
	well-known. Recall that two consecutive covergents have distance
	$u_{t+1}/v_{t+1}-u_{t}/v_{t}=(-1)^{t}/(v_{t}v_{t+1})$.
	In particular convergents
	lie alternatingly on the left and the right of the limit,
	and the right inequality follows.
	The left estimate follows similarly incorporating 
	also that $u_{t+1}/v_{t+1}$ lies closer
	to $\xi$ than $u_{t}/v_{t}$, also a well-known fact. 
	We have to show the second claim involving \eqref{eq:formu2}.
	Define the convex body
	\[
	K= [-v_{t}, v_{t}] \times [-v_{t+1}^{-1}, v_{t+1}^{-1}],
	\] 
	and the lattice
	\[
	\Lambda_{\xi}= \{ (m,m\xi-n)\in\mathbb{R}^{2}: m,n \in \mathbb{Z}  \}.
	\]
	The lattice has determinant $\det(\Lambda_{\xi})=1$ 
	and the rectangular convex body has volume $\rm{vol}(K)=4v_{t}/v_{t+1}$. 
	We may assume $\xi\in (0,1)$ and thus $0<u_{l}<v_{l}$ for all large $l$.
	Since we know from the theory of continued
	fractions and \eqref{eq:formu} that 
	\[
	\vert u_{t+1}\xi-v_{t+1}\vert < \vert u_{t}\xi-v_{t}\vert \leq \frac{1}{v_{t+1}},
	\]
	we see that the point $(u_{t}, u_{t}\xi-v_{t})$
	lies in $\Lambda_{\xi}\cap K$.
	Hence the first
	successive minimum $\lambda_{1}(K,\Lambda_{\xi})$
	of $\Lambda$ with respect to $K$ satisfies $\lambda_{1}(K,\Lambda_{\xi}) \leq 1$.
	By Minkoswki's second Convex Body Theorem \eqref{eq:diesegl}
	we conclude that the second 
	successive minimum satisfies
	\[
	\lambda_{2}(K,\Lambda_{\xi})\geq \frac{2\det(\Lambda_{\xi})}{V(K)\lambda_{1}(\Lambda_{\xi},K)}\geq \frac{v_{t+1}}{2v_{t}}.
	\] 
	Thus there is no lattice point linear independent from
	$(u_{t}, u_{t}\xi-v_{t})$ in the region
	\[
	\frac{v_{t+1}}{2v_{t}}\cdot K= [-\frac{v_{t+1}}{2},\frac{v_{t+1}}{2}]
	\times [-\frac{1}{2v_{t}},\frac{1}{2v_{t}}]. 
	\]
	In other words, for any $(r,s)$ linearly
	independent from $(u_{t},v_{t})$ and with $\vert s\vert\leq v_{t+1}/2$ we have
	$\vert s\xi-r\vert\geq 1/(2v_{t})$. The proof of the second
	claim is finished.
\end{proof}

For the density result in Theorem~\ref{thm11} we utilize the following lemma.

\begin{lemma} \label{kokomo}
	Let $C\subseteq \mathbb{R}^{d}$ be a Cantor set. Assume
	for every function $\Psi:\mathbb{N}\to \mathbb{R}_{>0}$ 
	there exists $\underline{\xi}\in C$ so that
	\[
	\Vert \underline{\xi}-\underline{p}/q\Vert \leq \Psi(q)
	\] 
	has infinitely many solutions $\underline{p}/q\in \mathbb{Q}\cap C$.
	Then the set of $\underline{\xi}$ with the same property is dense in $C$. 
\end{lemma}

\begin{proof}[Proof of Lemma~\ref{kokomo}]
	Let 
	$\underline{\xi}\in C$ arbitrary with
	address $(\omega_{1},\omega_{2},\ldots)$ and write $\pi_{i}=f_{\omega_{i}}$ so that
	$\xi=\lim_{k\to\infty} \pi_{1}\circ \ldots \circ\pi_{k}(\underline{0})$. For given $\Psi$ and
	$\epsilon>0$, we construct
	a $\Psi$-approximable point $\underline{\xi}_{\epsilon}$ with $\Vert \underline{\xi}_{\epsilon}-\underline{\xi}\Vert <\epsilon$. For again $\tau\in(0,1)$ the largest absolute value of the contraction factors and $diam$ the diameter of $C$,
	take an integer $u$ large enough so that $\tau^{u} diam < \epsilon$. 
	Then consider the function $\Phi(t):= \Psi(Nt)\cdot N^{-1}$
	for a large integer $N$ dependent only on $u$ (and thus $\epsilon$) to be
	chosen later. By assumption there exists
	$\underline{\zeta}\in C$ that is $\Phi$-approximable.
	Define $\xi_{\epsilon}= \pi_{1}\circ \pi_{2}\circ\cdots \pi_{u}(\underline{\zeta})$.
	Then we have
	\[
	\Vert \underline{\xi}_{\epsilon}-\underline{\xi}\Vert \leq \tau^{u} diam< \epsilon, 
	\]
	as the addresses of $\underline{\xi}$ and $\underline{\xi}_{\epsilon}$ 
	coincide up to the $u$-th place. On the other hand, we claim that
	$\underline{\xi}_{\epsilon}$ is $\Psi$-approximable if $N$ was chosen large enough.
	To see this first notice that as $\pi_{j}$ are linear maps with rational
	coefficients, also $T:=\pi_{1}\circ \pi_{2}\circ\cdots \pi_{u}$
	induces a linear, rational transformation 
	$T(\underline{y})=(A\underline{y}+\underline{b})/s$,
	with $A\in\mathbb{Z}^{d\times d},\underline{b}\in\mathbb{Z}^{d}$
	and $s\in\mathbb{N}$.
	It is not hard to see that if some $\underline{y}\in\mathbb{R}$ is $h$-approximable for some function $h(t)$, 
	then $T(\underline{y})$ (with $T$ as above) is $Nh(t/N)$-approximable,
	for sufficiently large $N$ depending only on $s$. 
	This is readily derived from the facts that 
	$\Vert T(\underline{y})-T(\underline{p}/q)\Vert \leq \Vert \underline{y}-\underline{p}/q\Vert$ (since $T$ is a contraction)
	and $den(T(\underline{p}/q))\leq sq$, where $den$ denotes the common denomintor of a rational vector. 
	Application to $h(t)=\Phi(t)$ and $\underline{y}=\underline{\zeta}$ starting with some corresponding
	suitable value $N=N(\epsilon)$
	yields that $\underline{\xi}_{\epsilon}=T(\underline{\zeta})$ is $N\Phi(t/N)=\Psi(t)$-approximable.
\end{proof}

Now we can prove Theorem~\ref{thm11}.

	\begin{proof}[Proof of Theorem~\ref{thm11}]
	We recursively define a fast growing lacunary sequence $(a_{n})_{n\geq 1}$  of positive integers that will induce $\underline{\xi}\in C$
	satisfying \eqref{eq:missing11}, as carried out below. Assume for the moment this sequence is fixed.
	Write $b_{1}=a_{1}-1$ and $b_{l}=a_{l}-a_{l-1}-1$ for integers $l\geq 2$.
	Let $f=f_{1},g=f_{2}$ be two different 
	contractions in our IFS and without loss of generality assume
	$f$ is the map whose fixed point $\underline{\alpha}=\underline{\alpha}_{j}$ 
	is rational. 
	Let $\tau_{i}$ for 
	$i=1,2$ be the contraction factors of $f$ and $g$ respectively
	and let $\tau=\max_{i=1,2} \vert \tau_{i}\vert<1$.
	We define $\underline{\xi}$ 
	by
	\begin{equation} \label{eq:gx}
	\underline{\xi}:=
	\lim_{i\to\infty} f^{b_{1}}\circ g\circ f^{b_{2}}\circ g\circ \cdots
	\circ	f^{b_{i}}(\underline{0}).
	\end{equation}
	This means that the contraction is $g$ at the 
	$a_{i}$-th positions and $f$ otherwise.
	Assume we have 
	chosen $a_{1}<a_{2}<\ldots<a_{k}$ and want to construct the remaining sequence
	with the property stated in the theorem. 
	Take 
	\begin{equation} \label{eq:hophas}
	\underline{\theta}_{k}:=
	f^{b_{1}}\circ g\circ f^{b_{2}}\circ g\circ \cdots \circ f^{b_{k}}\circ g 
	\circ	f^{\infty}(\underline{0}).
	\end{equation} 
	This vector clearly lies in $C$. Then clearly the suffix vector $f^{\infty}(\underline{0})$ 
	of $\underline{\theta}_{k}$ is the fixed point of $f$, so that 
	$f^{\infty}(\underline{0})=f(\underline{\alpha})=\underline{\alpha}$ and recall $\underline{\alpha}$
	is assumed to be rational. Since both $f,g$ are rational-preserving, we see that 
	$\underline{\theta}_{k}=f^{b_{1}}\circ g\circ f^{b_{2}}\circ g\circ \cdots \circ f^{b_{k}}\circ g(\underline{\alpha})$ is also a rational vector (if
	the IFS is affine then with denominator
	$\ll N^{da_{k}}$ 
	if $N$ is the largest denominator among the contraction
	factors, but we will not need this). 
	Write $\underline{\theta}_{k}= \underline{p}_{k}/q_{k}=(p_{k,1},\ldots, p_{k,d})/q_{k}$	in lowest terms, with $q_{k}\ll N^{da_{k}}$.
	Next, we claim that
	\begin{equation} \label{eq:gg1}
	\Vert \underline{p}_{k}/q_{k}-\underline{\xi} \Vert \leq diam \cdot \tau^{a_{k+1}}
	\end{equation}
	where $diam$ denotes the diameter of our compact Cantor set $C$.
	Observe
	\[
	\Vert \underline{p}_{k}/q_{k}-\underline{\xi} \Vert = 
	\Vert f^{b_{1}}\circ g\circ  \cdots \circ g\circ f^{b_{k+1}}(\underline{\alpha})- 
	f^{b_{1}}\circ g\circ  \cdots \circ g\circ f^{b_{k+1}}(\underline{\beta})  \Vert
	\leq \tau^{a_{k+1}}\Vert \underline{\alpha}-\underline{\beta}\Vert
	\]
	for
	\[
	\underline{\beta}= 
	\lim_{L\to\infty} f^{b_{k+1}}\circ g\circ \cdots \circ f^{b_{L}}(\underline{0})
	\]
	some well-defined limit
	in $C$, since $f,g$ are
	contractions with factor at most $\tau$ and there are a total
	of $a_{k+1}$ contractions applied. It then suffices to notice that 
	$\Vert \underline{\xi}-\underline{\beta}\Vert \leq diam$ as they both belong to $C$, to derive \eqref{eq:gg1}.
	From \eqref{eq:gg1} we see that $\Vert \underline{\xi}-\underline{p}_{k}/q_{k}\Vert$ can
	be made arbitrarily small when taking $a_{k+1}$ large enough, which
	is all we need for the sequel.
	
	Since $\tau<1$ by definition, we may 
	assume $a_{k+1}$ is large enough that
	$p_{k,i}/q_{k}$ is
	a convergent (not necessarily in lowest terms!) to $\xi_{i}$ for
	any $1\leq i\leq d$. 
	Let $s_{k,i}/t_{k,i}$ be the preceding convergent and 
	$y_{k,i}/z_{k,i}$ be the convergent following
	$p_{k,i}/q_{k}$, for $1\leq i\leq d$. We have to justify that
	this is well-defined, in the sense that  $s_{k,i}/t_{k,i}$
	and $y_{k,i}/z_{k,i}$
	do only depend on $a_{k+1}$ but
	not on the exact choices of $a_{k+2}, a_{k+3},\ldots$,
	for any given sufficiently large $a_{k+1}$ and much larger $a_{k+2}$.
	It is clear for the preceding convergent $s_{k,i}/t_{k,i}$.
	For $y_{k,i}/z_{k,i}$, notice that 
	when we choose $a_{k+2}$ much larger than $a_{k+1}$ as well then we
	can guarantee that $p_{k+1,i}/q_{k+1}$ is a convergent to $\xi_{i}$ 
	as well, by the same argument as above. 
	Thereby we can reconstruct the convergents of $\xi_{i}$ up to $\theta_{k+1,i}=p_{k+1,i}/q_{k+1}$. 
	Now $z_{k,i}\leq q_{k+1}$ by definition of $y_{k,i}/z_{k,i}$ as the subsequent convergent
	of $p_{k,i}/q_{k}$. Hence indeed
	once we have chosen large $a_{k+1}$ and assume $a_{k+2}$
	exceeds $a_{k+1}$ by a large amount, everything is well-defined. 
	Since by Lemma~\ref{lele} 
	\[
	z_{k,i}\geq \frac{1}{2q_{k}\vert \xi_{i}-\frac{p_{k,i}}{q_{k}}\vert}, 
	\]
	and $q_{k}$ is fixed, we infer from \eqref{eq:gg1}
	that we can make $z_{k,i}$ arbitrarily large
	by choosing $a_{k+1}$ accordingly large (note that for this argument
	we do not require $p_{k,i}/q_{k}$ to be in lowest terms, the same applies
	to any argument below).
	
	By Lemma~\ref{lele}
	we know that
	\[
	\vert \xi_{i}-\frac{s_{k,i}}{t_{k,i}}\vert \geq \frac{1}{2t_{k,i}q_{k}}, \qquad
	\vert \xi_{i}-\frac{p_{k,i}}{q_{k}}\vert \geq \frac{1}{2z_{k,i}q_{k}},
	\]
	and for any ratinal number $r_{i}/s\neq p_{k,i}/q_{k}$ with $0<s<z_{k,i}/2$ we have
	\[
	\vert s\xi_{i}-r_{i}\vert \geq \frac{1}{2q_{k}}. 
	\]
	So let $Q_{k}=z_{k,i}/3$. Thus for any $r_{i}/s$ as above
	\begin{equation} \label{eq:tisch}
	\vert \xi_{i}-\frac{r_{i}}{s}\vert \geq  \frac{1}{2sq_{k}}\geq  \frac{1}{2q_{k}} \cdot \frac{1}{z_{k,i}}. 
	\end{equation}
	Recall $q_{k}$ is fixed so $1/(2q_{k})$ 
	is a constant as well. On the other hand 
	\begin{equation} \label{eq:bett}
	\frac{\Phi(Q_{k})}{Q_{k}}=3\cdot \frac{\Phi(z_{k}/3)}{z_{k,i}}.
	\end{equation}
	Now by assumption $\Phi$ tends to $0$, and we have observed above
	that we may choose $a_{k+1}$ and consequently
	$Q_{k}$ arbitrarily large. 
	Thus we may choose $a_{k+1}$ large enough so that $\Phi(Q_{k})< 1/(6q_{k})$,
	and \eqref{eq:tisch} and \eqref{eq:bett} imply
	\[
	\vert \xi_{i}-\frac{r_{i}}{s}\vert > \frac{\Phi(Q_{k})}{Q_{k}}
	\]
	for any $r_{i}/s\neq p_{k,i}/q_{k}=\theta_{k,i}$ with $s<Q_{k}$. 
	Since this holds for any $1\leq i\leq d$ and
	$\underline{\theta}_{k}\in C$, this means
	that any $\underline{r}/s\notin C$ with $s<Q_{k}$ satisfies 
	\[
	\Vert \underline{\xi}-\underline{r}/s \Vert 
	\geq \vert \xi_{i}-r_{i}/s\vert >  \frac{\Phi(Q_{k})}{Q_{k}}.
	\]
	The desired property holds for $Q=Q_{k}$.
	Repeating this process, 
	we get a sequence of values $(Q_{n})_{n\geq 1}\to\infty$ with
	property \eqref{eq:missing11} for $\underline{\xi}$ defined in \eqref{eq:gx}. 
	
	To finish the proof, we show that all vectors 
	$\underline{\xi}=(\xi_{1},\ldots,\xi_{d})$ derived from our 
	construction are irrational, and that we can
	construct uncountably
	many distinct ones among them. The density is then also
	implied by Lemma~\ref{kokomo}. Clearly the
	the method is flexible enough to provide uncountably
	many formal elements $f^{b_{1}}\circ g\circ f^{b_{2}}\circ g\circ \cdots$
	with the given property. If we assume any element has
	at most countably many addresses, we are done. If not, by hypothesis
	we may assume $f,g$ are one-to-one.
	This case requires more work.
	We first show that any $\underline{\xi}$ constructed above is 
	irrational.
	Observe that by the fast growth of the $a_{i}$, we may assume
	that the approximating rational
	vectors
	$\underline{\theta}_{k}=\underline{p}_{k}/q_{k}$ 
	to $\underline{\xi}$ satisfy
	\begin{equation} \label{eq:ivo}
	\Vert \underline{\xi}-\underline{p}_{k}/q_{k}\Vert < \frac{1}{2q_{k}^{3}}.
	\end{equation}
	In particular any $\theta_{k,i}=p_{k,i}/q_{k}$ is
	a convergent to $\xi_{i}$ (not necessarily in lowest terms).
	Note that the denominators $q_{k}$ tend to infinity,
	since we divide by
	some integer in each contraction step.
	If some coordinate $\xi_{i}=r_{i}/s_{i}$ were rational, then clearly any rational number $p_{k,i}/q_{k}$
	has distance $\geq (s_{i}q_{k})^{-1}\gg q_{k}^{-1}$ unless $\xi_{i}=p_{k,i}/q_{k}$. 
	Hence, in view of \eqref{eq:ivo}, if $\xi_{i}$ was rational,
	then we must have $\xi_{i}=p_{k,i}/q_{k}$ for all large $k$. In particular
	$\theta_{k,i}= \theta_{k+1,i}$ for large $k$. Hence, if all coordinates
	$\xi_{i}$ are rational, by increasing $k$ if necessary we have the identity 
	for all coordinates $i=1,2,\ldots,d$, that is $\underline{\theta}_{k}=\underline{\theta}_{k+1}=\underline{\xi}$ 
	for any large $k$. 
	Since the addresses of $\underline{\theta}_{k}, \underline{\theta}_{k+1}$ 
	share the same prefix up to $f^{b_{k}}\circ g$ and $f,g$ are one-to-one by assumption, applying
	inverse functions repeatedly we infer an identity
	\begin{equation} \label{eq:fixes}
	f^{b_{k+1}}\circ g(\underline{\alpha})= \underline{\alpha}, \qquad \qquad \underline{\alpha}=f^{\infty}(\underline{0}).
	\end{equation}
	Recall $f$ fixes $\underline{\alpha}$. In fact,
	$\underline{y}=\underline{\alpha}$ is the only solution to $f(\underline{y})=\underline{\alpha}$ since $f$ is one-to-one.
	Now note that if all contractions of the IFS fix the same element $\underline{\alpha}\in\mathbb{Q}^{d}$,
	it is easy to check that the Cantor set collapses to a single point
	$C=\{\underline{\alpha}\}$, but then $\underline{\alpha}$ 
	has uncountably many addresses, which we have dealt with before. 
	Thus we may assume
	$g$ does not fix $\underline{\alpha}$, i.e. $g(\underline{\alpha})\neq \underline{\alpha}$. Then
	by the above observation that the preimage $\underline{\alpha}$
	under $f$ is only $\underline{\alpha}$, we see
	that \eqref{eq:fixes} cannot hold for any choice 
	of $b_{k+1}$.
	Thus we have confirmed that any $\underline{\xi}$ in our construction is irrational.
		Lemma~\ref{kokomo} and its proof show that 
	by taking finite forward orbit $\pi_{1}\circ\cdots \circ \pi_{m}(\underline{\xi})$ 
	of some $\underline{\xi}$ under the IFS essentially
	preserves the property, so we obtain
	a countably infinite set of suitable $\underline{\xi}$ that is dense in $C$.
	We only sketch the proof of why the set is actually uncountable and omit
	rigorous calculations.
	Assume $\underline{\xi}^{1}, \underline{\xi}^{2}$ are 
	two $\underline{\xi}$ as in \eqref{eq:gx} arising from sequences
	$(a_{n}^{1})_{n\geq 1}$ and $(a_{n}^{2})_{n\geq 1}$, 
	respectively. We claim that if they are ordered
	$a_{1}^{1}< a_{1}^{2} <a_{2}^{1} < a_{2}^{2}< a_{3}^{1}<\cdots$ with very large gaps
	between two consecutive elements, then $\underline{\xi}^{1}\neq \underline{\xi}^{2}$. If true
	this clearly
	implies we get an uncountable family of suitable $\underline{\xi}$. 
	The claim is obvious if there is some coordinate $i$ for which $\xi_{i}^{1}$
	is rational and $\xi_{i}^{2}$ is irrational, or vice versa. Thus, as we
	have shown above that both
	vectors $\underline{\xi}^{1}, \underline{\xi}^{2}$ are irrational,
	we may assume there is an index $i$ with both 
	 $\xi_{i}^{1},\xi_{i}^{2}$ irrational. 
	To show the claim,
	observe that the respective rational approximations 
	$\theta_{k,i}^{j}=p_{k,i}^{j}/q_{k,i}^{j}$ as in \eqref{eq:hophas}
	are again very good approximating convergents to $\xi_{i}^{j}$. 
	If we choose the gaps between two consecutive elements of
	$a_{1}^{1},a_{1}^{2},a_{2}^{1},a_{2}^{2},\ldots$ 
	sufficiently large in each step,
	then we will have $q_{k,i}^{1 \prime}<q_{k,i}^{2 \prime}<q_{k+1,i}^{1 \prime}$ 
	for any $k\geq 1$,
	where $p_{k,i}^{j \prime}/q_{k}^{j \prime}$ denotes $p_{k,i}^{j}/q_{k}^{j}$ written in lowest terms.
	Here we use that the denominators
	of convergents grow fast when the approximation is good, a well-known
	fact from the theory of continued fractions, see \eqref{eq:formu}.
	In fact, with a proper choice the 
	convergent $y_{k,i}^{2}/z_{k,i}^{2}$ of $\underline{\xi}^{2}$
	following $p_{k-1,i}^{2}/q_{k-1}^{2}$ will still have larger denominator
	than $q_{k,i}^{1}$. Hence $p_{k,i}^{1}/q_{k,i}^{1}$ is a convergent 
	to $\xi_{i}^{1}$ but no convergent
	to $\xi_{i}^{2}$, thus the continued fraction expansion 
	of $\xi_{i}^{1}$ and
	$\xi_{i}^{2}$ are not the same, consequently
	$\xi_{i}^{1}\neq \xi_{i}^{2}$ and $\underline{\xi}^{1}\neq \underline{\xi}^{2}$.   
\end{proof}

We turn towards the proof of  Theorem~\ref{bfr}.
For convenience, we
will use the framework of parametric geometry of
numbers introduced in~\cite{ss}, essentially
a parametric logarithmic version of Minkowski's Second Lattice Point
Theorem. We only need the special case of approximation to a 
single number, which corresponds to $n=2$ in the notation 
of~\cite{ss}. This means essentially we use a parametric, logarithmic version
of formula \eqref{eq:diesegl} above. 
Concretely, for given $\xi\in\mathbb{R}$, similar as for the proof
of Lemma~\ref{lele} consider the lattice $\Lambda_{\xi}$ and the family of convex
bodies $K(T)$ parametrized by $T\geq 1$ of the form
\begin{equation} \label{eq:1dimfall}
\Lambda_{\xi}= \{ (m,m\xi-n)\in\mathbb{R}^{2}: m,n\in\mathbb{Z}\}, \qquad
K(T)= [-T,T]\times [-T^{-1}, T^{-1}].
\end{equation} 
Observe that a point in $\Lambda_{\xi}\cap K(T)$ corresponds to the 
system of inequalities
\[
\vert m\vert \leq T, \qquad \vert m\xi-n\vert \leq T^{-1}.
\]
Denote by $\lambda_{1,\xi}(T), \lambda_{2,\xi}(T)$ the successive minima
of $K(T)$ with respect to the lattice $\Lambda_{\xi}$, that is 
$\lambda_{i,\xi}$ is the minimum
value for which $\lambda_{i,\xi}(T)K(T)$ contains $i$
linearly independent points of $\Lambda_{\xi}$, for $i\in\{1,2\}$.
According to Minkowski's Second Convex Body Theorem
\eqref{eq:diesegl}, in view of $\rm{vol}(K(T))=4$ and $\det \Lambda_{\xi}=1$,
we infer
\[
\frac{1}{2}=2 \frac{\det \Lambda_{\xi}}{\rm{vol}\; K(T)}\leq
\lambda_{1,\xi}(T)\lambda_{2,\xi}(T)\leq 
4\frac{\det \Lambda_{\xi}}{\rm{vol}\; K(T)}=1.
\]
If we let $t=\log T$ and $L_{i,\xi}(t)= \log \lambda_{i,\xi}(T)$ for
$i\in\{1,2\}$, then we obtain
\begin{equation} \label{eq:minko}
- \log 2 \leq L_{1,\xi}(t)+L_{2,\xi}(t)\leq 0, \qquad\qquad t\geq 0.
\end{equation}
The functions $L_{i,\xi}(t)$ are piecewise linear with slopes among
$\{-1,1\}$. More precisely, if $(m_{i},n_{i})\in\mathbb{Z}^{2}$ with $m_{i}>0$
are the vectors
realizing the $i$-th minimum at given position $t=\log T$
for $i\in\{1,2\}$, then
\[
L_{i,\xi}(t)=L_{(m_{i},n_{i})}(t), \qquad i\in\{ 1,2\},
\]
where for $(m,n)\in\mathbb{Z}^{2}$ we denoted
\begin{equation} \label{eq:yy}
L_{(m,n)}(t):= \max \{ \log m-t, 
\log \vert m\xi-n\vert + t\}.
\end{equation}
In particular $L_{(m,n)}(t)$ has its minimum at the point position
$(t_{0},L_{(m,n)}(t_{0}))$ with
$t=t_{0}= (\log m-\log \vert m\xi-n\vert)/2$ where the expressions in the
maximum coincide. 
This setup will suffice for our purpose to prove Theorem~\ref{bfr},
we refer to~\cite{ss} for parametric geometry of numbers with
respect to simultaneous rational approximation to 
$\xi_{1},\ldots,\xi_{d}$.

\begin{proof}[Proof of Theorem~\ref{bfr}]
	First consider the special case $C=C_{b,W}$. 
	Fix any
	$p/q\in \mathbb{Q}\setminus C$. Let $Q=2bq$.
	Let further $\xi\in C\setminus \mathbb{Q}$
	be given. Note we are in the situation of Theorem~\ref{known}.	
	Application of Theorem~\ref{known} to $Q\geq 2b\geq 1$ yields that
	the system
	\[
	1\leq m\leq b^{Q^{\Delta}},\qquad \quad
	\vert m\xi-n\vert\leq \frac{b}{Q}
	\]
	has a solution in positive integers $m,n$ such that 
	$\frac{n}{m}\in\mathbb{Q}\cap C$. 
	Write $\vert m\xi-n\vert= \sigma/Q$ 
	with $\sigma\in[0,b]$. Since $\xi$
	is irrational in fact $\sigma>0$.
	Transition to logarithmic scale yields that the 
	minimum of the function $L_{(m,n)}(t)$ is attained for 
	\[
	t=t_{0}:= \frac{\log m+\log Q-\log \sigma}{2},
	\]
	and at this position we have
	\[
	L_{(m,n)}(t_{0})= \log m-t_{0}=\frac{\log m-\log Q+\log \sigma}{2}.
	\]
	Clearly $(p,q)\in\mathbb{Z}^{2}$ we started with
	is linearly independent to $(m,n)\in\mathbb{Z}^{2}$, since this just rephrases
	as $\frac{m}{n}\neq \frac{p}{q}$, which is true since $\frac{m}{n}\in C$
	but $\frac{p}{q}\notin C$.
	Hence by \eqref{eq:minko}
	we have
	\[
	L_{(p,q)}(t_{0})\geq \frac{\log Q-\log m-\log \sigma}{2}-\log 2,
	\]
	as the reverse estimate would imply 
	$L_{1,\xi}(t_{0})+L_{2,\xi}(t_{0})\leq L_{(p,q)}(t_{0})+L_{(m,n)}(t_{0})<-\log 2$, contradicting \eqref{eq:minko}. In view of the definition
	of $L_{(p,q)}$ in \eqref{eq:yy},
	equivalently at least one of the inequalities 
	\[
	\log q-t_{0} \geq \frac{\log Q-\log m-\log \sigma}{2}-\log 2
	\]
	and
	\[
	\log \vert q\xi-p\vert +t_{0} \geq \frac{\log Q-\log m-\log \sigma}{2}-\log 2
	\]
	holds. Thus, if the first estimate is violated, which just becomes
	\begin{equation} \label{eq:gsd}
	q\leq \frac{Q}{2\sigma},
	\end{equation}
	then the second one is correct which yields
	\[
	\vert q\xi-p\vert \geq \frac{1}{2m}.
	\]
	However, \eqref{eq:gsd} is satisfied
	since 
	$q= \frac{Q}{2b}\leq \frac{Q}{2\sigma}$, 
	therefore we conclude 
	\begin{equation} \label{eq:tadaa}
	\vert q\xi-p\vert \geq \frac{1}{2m}\geq \frac{b^{-Q^{\Delta}}}{2}
	= \frac{b^{-(2b)^{\Delta}q^{\Delta}}}{2}.
	\end{equation}
	Notice that the bound 
	is independent from the choice of $\xi\in C\setminus \mathbb{Q}$. Assume
	the distance $d(C,p/q)$ for some $p/q\in \mathbb{Q}\setminus C$ was smaller
	than the right hand side of \eqref{eq:tadaa} divided by $q$.
	Then by compactness
	of $C$ there is $\xi_{0}\in C$ realizing this distance 
	$d(C,p/q)=d(\xi_{0},p/q)$.
	Clearly $\xi_{0}\notin C\setminus \mathbb{Q}$, since for these numbers we have 
	\eqref{eq:tadaa}. However, since $C\setminus \mathbb{Q}$ 
	is dense in $C$, we
	get a contradiction to \eqref{eq:tadaa} anyway by choosing 
	$\xi\in C\setminus \mathbb{Q}$ 
	sufficiently close to $\xi_{0}$. Hence the right hand side in \eqref{eq:tadaa}
	divided by $q$ is a lower bound
	for $d(C,p/q)$, which means \eqref{eq:null} is true. 
	Since $p/q\in \mathbb{Q}\setminus C$ and
	$\xi\in C\setminus \mathbb{Q}$ were chosen arbitrary,
	we readily infer \eqref{eq:0} from \eqref{eq:tadaa} as well.
	This finishes the proof of the special case $C=C_{b,W}$.
	
	For the general case of any monic, rationally generated Cantor set $C$ with 
	open set condition, by
	Theorem~\ref{argell} for any $\xi\in C\setminus \mathbb{Q}$ the estimates
	\[
	1\leq m\leq e^{Q^{\Delta}}, \qquad \vert \xi-\frac{m}{n}\vert\leq \frac{K}{Q} 
	\]
	admit a solution $m/n\in C$ for any large $Q$
	and some constant $K$. For given $p/q\in \mathbb{Q}\setminus C$, 
	we proceed as above for $Q=2Kq$ to obtain
	\[
	\vert q\xi-p\vert \geq \frac{1}{2m}\geq \frac{e^{-Q^{\Delta}}}{2}
	=\frac{e^{-(2K)^{\Delta}q^{\Delta}}}{2},
	\]
	and the claim follows again by choosing $\rho>(2K)^{\Delta}$ 
	sufficiently large to guarantee
	\[
	e^{-\rho q^{\Delta}}< \frac{e^{-(2K)^{\Delta}q^{\Delta}}}{q}
	\] 
	simultaneously for any $q\geq 1$.
\end{proof}

A similar strategy is applied for the proof 
of Theorem~\ref{expon} below,
indeed the right claim of \eqref{eq:e1} rewritten as $\lambda_{ext}(\xi)\leq 1/\widehat{\lambda}_{int}(\xi)$ resembles the assertion of
Theorem~\ref{bfr}.

\subsection{Proofs from Section~\ref{4}}

We start this section with the proof of the counting results.
For this employ Proposition~\ref{ppr} above.

\begin{proof}[Proof of Theorem~\ref{yzok}]
	Let $l\geq 0$ be an integer.
	Assume there are more than $J^{l}$ distinct rational vectors 
	in $C$ each with common denominator at most $N$.
	Then by Proposition~\ref{ppr}
	there are two vectors $\underline{\alpha}, \underline{\alpha}^{\prime}$ 
	among them that differ
	by at most $\Vert \underline{\alpha}-\underline{\alpha}^{\prime}\Vert\leq diam \cdot\tau^{l}$. On the other hand, if we write
	$\underline{\alpha}^{\prime}=\underline{p}/q, 
	\underline{\alpha}^{\prime}=\underline{r}/s$, then
	\begin{equation}  \label{eq:wird}
	\Vert \underline{\alpha}-\underline{\alpha}^{\prime}\Vert = \Vert \underline{p}/q-\underline{r}/s\Vert
	\geq \frac{1}{qs}\geq N^{-2}.
	\end{equation}
	We conclude $N\geq diam^{-1/2}\tau^{-l/2}$. In other words, if $N<diam^{-1/2}\tau^{-l/2}$
	then there are at most $J^{l}=\tau^{-Dl}$ rational vectors with common
	denominator at most $N$. Now for given $N$, let $l$ be the unique integer
	with $diam^{-1/2}\tau^{-l/2+1}\leq N<diam^{-1/2}\tau^{-l/2}$. Then by the above argument
	there are at most
	$J^{l}=\tau^{-Dl}<diam^{D}(N/\tau)^{2D}=J^{2}diam^{D}\cdot N^{2D}$ 
	rational vectors with common denominator at
	most $N$. 
	\end{proof}

The estimate extends to bound the cardinality of
vectors $(r_{1}/s_{1},\ldots,r_{d}/s_{d})$ with each denominator 
$\max_{1\leq j\leq d} s_{j}\leq N$.
The proof of Theorem~\ref{yok} employs Liouville's inequality.

\begin{proof}[Proof of Theorem~\ref{yok}]
Let $\underline{\alpha}=(\alpha_{1},\ldots,\alpha_{d}),\underline{\alpha}^{\prime}=
(\alpha_{1}^{\prime},\ldots,\alpha_{d}^{\prime})$ be distinct algebraic vectors 
with entries of degree at most $n$ and heights $H(\underline{\alpha})=\max H(\alpha_{j}),
	H(\underline{\alpha}^{\prime})=\max H(\alpha_{j}^{\prime})$ at most $N$, respectively. Since the
	vectors are distinct there is an index $j$ 
	with $\alpha_{j}\neq \alpha_{j}^{\prime}$.
	Then
	Liouville's inequality~\cite[Theorem~A.1, Corollary~A.2]{bugbuch} yields
	\[
	\Vert \underline{\alpha}-\underline{\alpha}^{\prime}\Vert\geq 
	\vert \alpha_{j}-\alpha^{\prime}_{j}\vert \gg_{n} H(\alpha_{j})^{-n}H(\alpha_{j}^{\prime})^{-n}\geq
	N^{-2n}. 
	\]
	Using this estimate in place of \eqref{eq:wird}, the claim 
	follows very similarly as Theorem~\ref{yzok}. 
\end{proof}

Next we show
Theorem~\ref{pbasic}.

\begin{proof}[Proof of Theorem~\ref{pbasic}]
	Let $(\omega_{1},\omega_{2},\ldots)$ be an address of $\underline{\xi}\in C$ and $\pi_{i}=f_{\omega_{i}}$, so that 
	$\underline{\xi}=\lim_{n\to\infty} \pi_{1}\circ \pi_{2}\cdots\circ \pi_{n}(\underline{0})$.
	Assume $(\omega_{1},\ldots,\omega_{l})$ is 
	the (possibly empty) preperiod and $(\omega_{l+1},\ldots,\omega_{u})^{\infty}$
	is the successive period. Let
	$\underline{\zeta}:=(\pi_{l+1}\circ \cdots \circ \pi_{u})^{\infty}(\underline{0})$ so that
	$\underline{\xi}=\pi_{1}\circ \cdots \circ \pi_{l}(\underline{\zeta})$.
	By construction $\underline{\zeta}=\pi_{l+1}\circ \cdots \circ \pi_{u}(\underline{\zeta})$. For simplicity write $C_{j}=A_{\omega_{j}}/q_{\omega_{j}}\in \mathbb{Q}^{d\times d}$	and $\underline{c}_{j}=\underline{b}_{\omega_{j}}/s_{\omega_{j}}\in\mathbb{Q}^{d}$, with $A_{j}, \underline{b}_{j}$ as in Definition~\ref{4def}, so that
	$\pi_{j}(\underline{y})= C_{j}\underline{y}+\underline{c}_{j}$.
	Therefore from the concatenation we obtain an identity $C_{l+1}C_{l+2}\cdots C_{u}\underline{\zeta}+\underline{c}^{\prime}= \underline{\zeta}$, for some
	$\underline{c}^{\prime}\in \mathbb{Q}^{d}$.
	Since $C_{j}$ induce contractions so does any product, so $1$ is not an eigenvalue
	and we see $\underline{\zeta}$ is given as $\underline{\zeta}=-(C_{u}C_{u-1}\cdots C_{l}-I_{d})^{-1}\underline{c}^{\prime}$. Hence $\underline{\zeta}\in\mathbb{Q}^{d}$ as $C_{j}\in\mathbb{Q}^{d\times d}$ 
	and $\underline{c}^{\prime}\in \mathbb{Q}^{d}$. Thus also $\underline{\xi}=\pi_{1}\circ \cdots \circ \pi_{l}(\underline{\zeta})$ is 
	rational since $\pi_{j}\in F$ are rational-preserving. 
	
	Conversely, assume the IFS is unimodular and take arbitrary 
	$\underline{\xi}=\underline{p}/q\in \mathbb{Q}^{d}\cap C$. 
	Again assume 
	$\pi_{1}\circ \pi_{2} \cdots(\underline{0})$ is any formal
	representation of $\underline{\xi}$. 
	Consider the sequence
	$\underline{a}_{0}=\underline{\xi},\underline{a}_{1}=\pi_{1}^{-1}(\underline{\xi}), \underline{a}_{2}=\pi_{2}^{-1}\circ \pi_{1}^{-1}(\underline{\xi}), \ldots$, which
	is well-defined as clearly $\pi_{j}$ are bijective if the IFS is unimodular. 
	Then  obviously $\underline{a}_{n}\in \mathbb{Q}^{d}\cap C$ for any $n\geq 0$. 
	Moreover, when building the inverses to derive $\underline{a}_{n+1}$ from
	$\underline{a}_{n}$, it follows from the rational IFS being unimodular that the 
	denominators that occur are divisors of $qS$ with $S:=\prod_{1\leq j\leq J} s_{j}$. Here we use that 
	when building the inverses $\pi_{j}^{-1}$ with Cramer's rule
	we do not get additional factors in the denominator 
	by unimodularity of the matrices, and the shift 
	vectors $\underline{b}_{j}/s_{j}$ can
	possible only cause a factor that divides $S$. 
	Since $C$ is compact with some finite 
	diameter $diam>0$, there are at most
	$(qS\cdot diam+1)^{d}\ll_{C} q^{d}$ rational vectors in $C$
	with this property. 
	Moreover, since $\underline{a}_{n}\in \mathbb{Q}^{d}\cap C$ with 
	denominator $\ll_{C} q$
	we infer from Theorem~\ref{yzok} and its proof
	that there are at most $\ll_{C} q^{D}$ such 
	vectors (we may remove the factor $2$ in the exponent since all denominators
	divide $q$, see the proof of Theorem~\ref{yzok}). 
	Hence the number is $\ll_{C} \min\{ q^{D},q^{d}\}$. 
	 By this finiteness, some rational vector
	must occur twice, that is $\underline{a}_{i}=\underline{a}_{j}$ for 
	some $i<j\ll_{C} \min\{ q^{D},q^{d}\}$. 
	This means
	$\underline{a}_{j}=\pi_{i+1}\circ \pi_{i+2}\circ \ldots \circ \pi_{j}(\underline{a}_{i})= \underline{a}_{i}$. 
	Hence
	\begin{eqnarray*}
		\underline{\xi}&= \pi_{1}\circ \cdots \circ \pi_{i}(\underline{a}_{i})=
		\pi_{1}\circ \cdots \circ \pi_{i}\circ (id)^{\infty}(\underline{a}_{i})\\
		&=\pi_{1}\circ \cdots \circ \pi_{i} 
		\circ (\pi_{i+1}\circ \pi_{i+2}\circ \ldots \circ \pi_{j})^{\infty}(\underline{a}_{i}) 
		\\ &=
		\pi_{1}\circ \cdots \circ \pi_{i} 
		\circ (\pi_{i+1}\circ \pi_{i+2}\circ \ldots \circ \pi_{j})^{\infty}(\underline{0}).
	\end{eqnarray*}
	In the last identity we used that $\pi_{i}\in F$ are contractions. Thus 
	$(\omega_{1},\ldots ,\omega_{i})$ composed with
	$(\omega_{i+1},\omega_{i+2},\ldots,\omega_{j})^{\infty}$ is a
	periodic address of $\underline{\xi}$.
\end{proof}

The estimate on rationals in $C$ whose denominator divides $q$ we used is 
potentially very
crude, we believe in fact $\log q$ should be the correct order of period lengths.

 \begin{proof}[Proof of Corollary~\ref{sophieg}]
 	For the first claim let $q_{i}$ be as in the theorem
 	and $Q=q_{1}q_{2}\cdots q_{v}$ their product.
 	We notice that for a given prime $q\nmid b$
 	it follows directly from~\cite[Lemma~3]{blo}
 	that $ord_{b}\bmod q^{n}= mq^{n-O(1)}\gg q^{n}$ for some $m$ diving
 	$q-1$ and all $n$.
 	Thus, since for $s=q_{1}^{\alpha_{1}}q_{2}^{\alpha_{2}}\cdots q_{v}^{\alpha_{v}}$ we have that $ord_{b}\bmod s$ is the lowest common multiple 
 	of the orders $ord_{b}\bmod q_{i}^{\alpha_{i}}$ over $1\leq i\leq v$ and
 	$q_{i}$ are distinct primes not dividing $b$,
 	we see that $ord_{b}\bmod s$ is divisible by $s/Q^{O(1)}$ 
 	and thus $ord_{b}\bmod s\gg s$. 
 	On the other hand, according to Proposition~\ref{pr} the period
 	length of $r/s\in C_{b,W}$ equals $ord_{b} \bmod s$, which
 	is $\ll_{C} s^{D}=s^{\Delta}$ by Theorem~\ref{pbasic}. 
 	For large $s$
 	we get a contradiction as $\Delta<1$. 
 	For the second claim,
 	again since the period length of $p/q$ is $ord_{b} \bmod q$, it divides
 	$\varphi(q)=q-1$. Since $(q-1)/2$ is prime the only possible divisors
 	are $\{1,2,(q-1)/2,q-1\}$. On the other hand, by Theorem~\ref{pbasic} the
 	period length is $\ll q^{D}=q^{\Delta}$. Since $\Delta<1$ for large $q$ the divisors $(q-1)/2,q-1$ can be ruled out,
 	so $q\vert b^{2}-1$ which only gives finitely many values of $q$ as well. 
 	In both claims, since $C_{b,W}\subseteq [0,1]$ this yields only finitely many 
 	potential rational numbers of the stated form.  
 	\end{proof}

As indicated above, the proof of Theorem~\ref{expon} is similar to the one of Theorem~\ref{bfr}.

\begin{proof}[Proof of Theorem~\ref{expon}]
	As in Lemma~\ref{lele} let $u_{t}/v_{t}$ be the convergents of $\xi\in S$. We only
	show the left inequality of \eqref{eq:e1}, the right is proved analogously.
	We may assume $\lambda_{int}(\xi)>1$ as otherwise the claim is trivial.
	For simplicity write $a= \lambda_{int}(\xi)$. Let $\epsilon>0$.
	Recall from the theory of continued fractions
	that any fraction
	$p/q$ that is no convergent to $\xi$ satisfies $\vert \xi-p/q\vert>(2q)^{-1}$.
	Thus, our assumption implies that
	there exist arbitrarily large indices $l$ so that the
	convergent $u_{l}/v_{l}$ lies in $S$ and has the property
	\begin{equation}  \label{eq:cofr0}
	v_{l}^{-a-\epsilon} \leq \vert v_{l}\xi-u_{l}\vert \leq v_{l}^{-a+\epsilon}.	
	\end{equation}
	By Lemma~\ref{lele} we have
	$v_{l+1}\asymp \vert v_{l}\xi-u_{l}\vert^{-1}$
	so for large $l$ we infer
	\begin{equation} \label{eq:cofr}
	v_{l}^{a-\epsilon} \ll v_{l+1}\ll v_{l}^{a+\epsilon}.
	\end{equation}
	Moreover again Lemma~\ref{lele} shows that for any $r/s\neq u_{l}/v_{l}$ 
	with $s<v_{l+1}/2$ we have
	$\vert s\xi-r\vert\geq v_{l}^{-1}/2$. In particular this estimate holds for any
	$r/s\notin S$ with $s<v_{l+1}/2$. Thus for $Q_{l}:=v_{l+1}/3$ 
	and any such $r/s\notin S$ with $s<Q$ we have
	\[
	\vert s\xi-r\vert \geq \frac{1}{2}v_{l}^{-1}.
	\]  
	On the other hand from \eqref{eq:cofr} we know
	$v_{l}\gg v_{l+1}^{1/(a+\epsilon)}\gg Q_{l}^{1/(a+\epsilon)}$. Thus the left hand side
	is $\gg Q_{l}^{-1/(a+\epsilon)}$ for any $r/s\notin S$ with $s<Q$.
	The left inequality in \eqref{eq:e1} follows since clearly $Q_{l}\to \infty$
	as $l\to\infty$ and	$\epsilon$ can be
	taken arbitrarily small. As indicated above, an analogous argument 
	starting with $a= \lambda_{ext}(\xi)$ yields the right estimate in \eqref{eq:e1}.
	
	In the case that $S=C$ is a Cantor set as in the last claim,
	we now provide the reverse of the left estimate of
	\eqref{eq:e1} in order to show \eqref{eq:foet}. 
	First assume $\lambda_{int}(\xi)=:a\geq 1$.
	We may again assume strict inequality $a>1$ as if otherwise
	$a=1$,
	the claimed estimate is a consequence of \eqref{eq:e1} and \eqref{eq:cons}.
	This again implies certain convergents satisfy $u_{l}/v_{l}\in C$ as above.
	Let $Q$ be an arbitrary, large real number. Let $l$ be the index such 
	that $v_{l}\leq Q<v_{l+1}$.
	Then by Lemma~\ref{lele} we have
	\begin{equation} \label{eq:abo}
	Q^{-1} >
	v_{l+1}^{-1} \geq \vert v_{l}\xi-u_{l}\vert.
	\end{equation}
	If $u_{l}/v_{l}\notin C$, then
	since $1>1/a$ we may simply choose $r/s=u_{l}/v_{l}$. 
	Now assume $u_{l}/v_{l}\in C$. 
	Then by \eqref{eq:cofr}, for large $l$
	additionally to \eqref{eq:abo} we have
	\[
	Q^{-1/(a+\epsilon)} \geq v_{l+1}^{-1/(a+\epsilon)}\gg
	v_{l}^{-1} \geq \vert v_{l-1}\xi-u_{l-1}\vert.
	\] 
	Combining,
	for any given natural number $N$ and any $1\leq n\leq N$, we have
	\begin{align*}
	\vert (v_{l}+nv_{l-1})\xi-u_{l}-nu_{l-1}\vert &
	\leq \vert v_{l}\xi-u_{l}\vert +
	n\vert v_{l-1}\xi-u_{l-1}\vert \\
	&\ll_{N} 
	\max \{Q^{-1}, Q^{-1/(a+\epsilon)} \}=  Q^{-1/(a+\epsilon)}.
	\end{align*}
	Now, from
	Lemma~\ref{fslemma} we see that for large $N=N(C)$ and some  
	$1\leq n\leq N$ the expression $(u_{l}+nu_{l-1})/(v_{l}+nv_{l-1})$ 
	does not belong to $C$. By the above estimate, as $\epsilon$ can be chosen
	arbitrarily small and $Q$ was chosen arbitrary,
	we verify the reverse inequality
	\[
	\widehat{\lambda}_{ext}(\xi)\geq \frac{1}{a}=\frac{1}{\lambda_{int}(\xi)},
	\]
	so by \eqref{eq:e1} we must have equality. Hence
	\eqref{eq:foet} holds in this case. 
	
	Finally assume $\lambda_{int}(\xi)<1$. Then we
	have $\vert q\xi-p\vert \gg q^{-1+\delta}$ for some $\delta>0$
	and any $p/q\in \mathbb{Q}\cap C$. On the other hand, by Dirichlet's
	Theorem we know that for any parameter $Q$ there is a rational
	number $p/q$ with $q\leq Q$ and $\vert q\xi-p\vert \leq Q^{-1}\leq q^{-1}$.
	Combining these observations,
	for large $Q$ these rationals $p/q$ must lie outside $C$. Hence
	$\widehat{\lambda}_{ext}(\xi)= \widehat{\lambda}(\xi)=1$, where
	the last identity is Khintchine's result already mentioned
	in \eqref{eq:cons}. Again we
	conclude \eqref{eq:foet}.
\end{proof}

The question arises if we can reverse the argument of the proof to infer
''good'' intrinsic rational approximations to $\xi$ from ''bad'' uniform extrinsic
approximation of $\xi\in C$, corresponding to a hypothetical inequality
like $\lambda_{int}(\xi)\geq 1/\widehat{\lambda}_{ext}(\xi)$ if
$\widehat{\lambda}_{ext}(\xi)$ is small enough (it is in general
false by a metrical argument
if $\widehat{\lambda}_{ext}(\xi)=1$). 
The problem is that in the lattice
point problem studied above, both successive minima may be
throughout realized by vectors
$(m_{i},n_{i})\in\mathbb{Z}^{2}$ that lead to rational numbers $m_{i}/n_{i}$ 
in $C$, for $i=1,2$, in which case can only infer $\lambda_{int}(\xi)\geq 1$. 
Nevertheless, as this seems to be rather exceptional situation that may not occur,
the estimate $\lambda_{int}(\xi)\geq 1/\widehat{\lambda}_{ext}(\xi)$ may be true if $\widehat{\lambda}_{ext}(\xi)<1$. 

\begin{problem}
	Assume $\widehat{\lambda}_{ext}(\xi)<1$. Does it
	follow that $\lambda_{int}(\xi)\geq 1/\widehat{\lambda}_{ext}(\xi)$?
\end{problem}

We turn towards the verification of Theorem~\ref{pthm}.

\begin{proposition} \label{ppp}
	Let $b\geq 3$ an integer and $W\subsetneq \{0,1,\ldots,b-1\}$.
	Assume 
	$\xi=(0.c_{1}c_{2}\cdots c_{k}\overline{c_{k+1}\cdots c_{N-1}})_{b}$ 
	is the base $b$ representation of 
	a rational number. If $c_{1},\ldots,c_{k+1}$ belong
	to $W$ and $c_{m}\notin W$ for some $m>k+1$, then $\xi\notin C_{b,W}$.  
\end{proposition}

\begin{proof}
	If $\xi$ has a unique base $b$ expansion then the claim follows from $c_{j}\notin W$.
	If the representation is not unique then there are two base $b$ representations
	of $\xi$, one ending in $\overline{0}$ and the other ending in $\overline{b-1}$. However,
	since $c_{k+1}\in W$ and $c_{j}\notin W$ for some $j>k+1$ and both letters
	appear in the period above, clearly $\xi$ is not of this form, contradiction.
\end{proof}

\begin{proof}[Proof of Theorem~\ref{pthm}]
	Clearly any $\xi\in\mathbb{Q}$ has ultimately periodic base $b$ 
	expansion,
	and the representation of $\xi$ in \eqref{eq:ain}
	via $\xi=p/q$ with $p,q$ in \eqref{eq:asinas} is carried
	out straightforwardly as in~\cite[Lemma~2.3]{1}. 
	Let $m=\phi(\xi)$. If $m\leq k+1$ so that the first digit 
	ouside $W$ belongs to the preperiod
	or equals the first period digit,
	then the estimate follows from the elementary estimate
	$m\leq k+1\ll \log q_{0}\ll q_{0}^{\Delta}$, as discussed in Section~\ref{rel}. 
	Thus we may assume $m>k+1$.
	Obviously $\xi$ has distance at most $b\cdot b^{-m}$
	from $r:=c_{0}b^{-1}+c_{1}b^{-2}+\cdots +c_{m-1}b^{-m}\in\mathbb{Q}$. 
	By definition of $m$ clearly 
	$r\in C_{b,W}$. On the other hand, by our assumption $c_{m}\notin W$ and by Proposition~\ref{ppp} we have
	$\xi\notin C_{b,W}$. Hence, by Theorem~\ref{bfr} the rational
	number $\xi$ has distance at least $b^{-\delta_{0}q_{0}^{\Delta}}$ from
	any element of $C=C_{b,W}$ for some uniform $\delta_{0}$, in particular from $r$.
	Comparison
	yields $\phi(\xi)=m\ll q_{0}^{\Delta}$. 
\end{proof}

\begin{proof}[Proof of Corollary~\ref{korolar}]
Obviously $W_{1}\subseteq W_{2}$, and  
if $W_{1}=\{0,1\ldots,b-1\}$ the claim is obvious.
Otherwise there is some 
$w\notin W_{1}$. If $w\in W_{2}$,
by taking
$W=\{0,1,\ldots,b-1\}\setminus \{w\}$,
from Theorem~\ref{pthm} we get $d^{\prime}\geq c(b)\cdot\phi^{1/\Delta}$. 
We get 
a contradiction 
by our assumptions on $\phi,\psi$ as soon as
$c_{1}\geq c(b)^{\Delta}$. Hence $w\notin W_{2}$,
and since $w\notin W_{1}$ was arbitrary the sets are equal.
\end{proof}

\vspace{1cm}

{\em The author thanks the referee for the careful reading and suggestions
	for improvement.
	The author further thanks Simon Kristensen, Diego Marques, Nikita Sidorov,
	Michael Coons and Igor Shparlinski 
	for help with references!}

\end{document}